\documentclass{amsart}
\usepackage{}
\usepackage{amsfonts}

\usepackage{amssymb}

\usepackage{mathrsfs}

\usepackage{amsmath}
\usepackage{mathrsfs}
\usepackage{url}
\newtheorem{theorem}{Theorem}[section]
\newtheorem{lemma}[theorem]{Lemma}
\newtheorem{corollary}[theorem]{Corollary}
\newtheorem{proposition}[theorem]{Proposition}

\theoremstyle{definition}

\theoremstyle{remark}

\numberwithin{equation}{section}

\begin{document}

\title{Embeddings of $P(\omega)/{\rm Fin}$ into Borel Equivalence Relations between $\ell_p$ and $\ell_q$}

\author{Zhi Yin}
\address{School of Mathematics, Nankai University, Tianjin, 300071, P.R.China}
\email{will.yin@hotmail.com}
\thanks{Research partially supported by the National Natural Science Foundation of China(Grant No. 11071129).}




\date{\today}


\begin{abstract}
We prove that, for $1 \le p<q<\infty$, the partially  ordered set $P(\omega)/{\rm Fin}$ can be embedded into Borel equivalence relations between $\mathbb{R}^\omega/\ell_p$ and $\mathbb{R}^\omega/\ell_q$. Since there is an antichain of size continuum in $P(\omega)/{\rm Fin}$, therefore there are continuum many incomparable Borel equivalence relations between $\mathbb{R}^\omega/\ell_p$ and $\mathbb{R}^\omega/\ell_q$.
\end{abstract}

\maketitle
\section{Introduction}

A topological space is \emph{Polish} if it is separable and completely metrizable. Let $X,~ Y$ be Polish spaces and $E,~F$ equivalence relations on $X, Y$ respectively. If there is a Borel map $\theta: X \to Y$ such that
\[
x E y \Longleftrightarrow \theta(x) F \theta(y)
\]
for all $x,y \in X$,  we say that $E$ is \emph{Borel reducible} to $F$, denoted  $E \le_B F$. If $E \le_B F$ but  $F \nleq_B E$, we say $E$ is \emph{strictly Borel reducible} to $F$, denoted $E<_BF$. If $E \le_B F$ and $F \le_B E$ , we say $E$ and $F$ are \emph{Borel equivalent} to each other, denoted $E \sim_B F$. If $E \nleq_B F$ and $F \nleq_B E$, we say $E$ and $F$ are \emph{Borel incomparable}. For more details about Borel reduction, one can see \cite{GB}.

Dougherty and  Hjorth \cite{Hj} showed that, for $1 \le p <q <\infty $, $\mathbb{R}^\omega/\ell_p <_B \mathbb{R}^\omega/\ell_q$. For $0<p \le 1$, $\mathbb{R}^\omega/\ell_p \sim_B \mathbb{R}^\omega/\ell_1$, a proof  of it can be found in   \cite{Ding}. All of $\ell_p (p \ge 1)$ are Borel reducible to $\mathbb{R}^\omega/\ell_\infty$ by  Rosendal \cite{Rosend}.

A question of  Gao \cite{GA} asking whether for $1 \le p <\infty$, $\mathbb{R}^\omega/\ell_p$ is the greatest lower bound of $\{\mathbb{R}^\omega/\ell_q: p<q\}$. Let $f: [0,1] \to \mathbb{R}^+$ be an arbitrary function,  M\'{a}trai  \cite{Ma} considered a kind of  $\ell_p$-like relation $\mathbf{E}_f$'s on $[0,1]^\omega$ by setting, for every $(x_n)_{n<\omega} , ~(y_n)_{n<\omega} \in [0,1]^\omega$,
\[
(x_n)\mathbf{E}_f (y_n) \Longleftrightarrow \sum_{n<\omega} f(|y_n-x_n|)< \infty.
\]
He studied the Borel reducibility of Borel equivalence relations of the form $\mathbf{E}_f$, and
 answer Gao's question in the negative by showing, for $1 \le p<q<\infty$, every linear subset of $P(\omega)/{\rm Fin}$ can be embedded into the set of equivalence relations between $\mathbb{R}^\omega/\ell_p$ and $\mathbb{R}^\omega/\ell_q$(See \cite{Ma}, Corollary 31).

 Another kind of $\ell_p$-like equivalence relations was introduced by Ding \cite{Ding3}.  Let $(X_n, d_n), n<\omega$ be sequence of  pseudo-metric spaces and $p \ge 1$. For $x,y \in \prod_{n<\omega}X_n, ~ (x,y) \in E((X_n, d_n)_{n<\omega}; p)\Leftrightarrow \sum_{n<\omega}d_n(x(n),y(n))^p<\infty$. He showed that Borel reducibility between $E((X_n)_{n<\omega}; p)$ equivalence relations is related to a notion of ``finitely H\"{o}lder($\alpha$) embeddability''. If $X$ is a separable Banach space with  norm $\| \cdot \|_X$ and $(X_n, d_n)=(X, \| \cdot \|_X)$ for each $n<\omega$,  denoted $E((X_n, d_n)_{n<\omega}; p)$ by $E(X, p)$, he presented many results on reducibility and nonreducibility between $E(L_r, p)$'s and $E(c_0, p)$'s for $r,p \in [1, \infty)$. For more details,  see \cite{Ding2}.

A well-known theorem of Parovi\v{c}enko (see e.g. \cite{Bella}) says that, every  Boolean algebra of size $\le \omega_1$ embeds into $P(\omega)/{\rm Fin}$. So assuming  CH (the continuum hypothesis), every Boolean algebra of size at most continuum can be embedded into $P(\omega)/{\rm Fin}$, therefore every partially  ordered set of size at most continuum can be embedded into $P(\omega)/{\rm Fin}$, this indicates $P(\omega)/{\rm Fin}$ is the most complicated partially ordered set of  size at most continuum. What happens if CH fails? We refer to \cite{Baum} for more results. Anyway, we know that there is an antichain of size continuum in $P(\omega)/{\rm Fin}$ under ZFC.

In \cite{Farah}, Farah introduced a family of Borel equivalence relations called $c_0$-equalities. Using a method from Louveau--Velickovic \cite{LV}, he proved that $P(\omega)/{\rm Fin}$ can be embedded into $c_0$-equalities. It follows that  there are continuum many Borel incomparable equivalence relations among them. On the other hand, we can see that all equivalence relations considered in  \cite{Ma} are pairwise Borel comparable. By contrast, for $1\le p <q <\infty$, several questions arise naturally:

\begin{enumerate}
\item[(i)]  Are there continuum many Borel incomparable equivalence relations between   $\mathbb{R}^\omega/\ell_p$ and $\mathbb{R}^\omega/\ell_q$?
\item[(ii)]  Does  $P(\omega)/{\rm Fin}$ can be embedded into equivalence relations between $\mathbb{R}^\omega/\ell_p$ and $\mathbb{R}^\omega/\ell_q$?
\end{enumerate}

We show that the answer for question (ii) is affirmative. In fact, we give a more stronger result.
\begin{theorem}
For $1 \le p <\infty$ and $U \in P(\omega)$, there is a Borel equivalence relation $\mathbf{E}_U$ such that  for any $q>p$, $\mathbb{R}^\omega/\ell_p \le_B \mathbf{E}_U \le_B \mathbb{R}^\omega/\ell_q$ and for $U, V \in P(\omega)$, we have
\[
U \subseteq^{*}V \Longleftrightarrow \mathbf{E}_U \le_B \mathbf{E}_V.
\]
\end{theorem}

Since there is an antichain of size continuum in $P(\omega)/{\rm Fin}$, we have the following corollary:
\begin{corollary}
For  $1 \le p <q<\infty$, there is a set of Borel equivalence relations
\[
\{\mathbf{E}_\xi:  \xi \in \{0,1\}^\omega\}
\]
such that $\mathbb{R}^\omega/\ell_p \le_B \mathbf{E}_{\xi} \le_B \mathbb{R}^\omega/\ell_q$, and for distinct $\xi,\zeta \in \{0,1\}^\omega$, we have $\mathbf{E}_{\xi}$ and $\mathbf{E}_{\zeta}$ are Borel incomparable.
\end{corollary}

{\it Notions.} Denote  $\mathbb{R}^+$ the set of nonnegative real numbers and  ${\rm Id}^\alpha$ the function $x^\alpha(x\in[0,1])$.  ${\rm Fin}$  stands for the collection of all finite sets of $\omega$. For $U, ~V \in P(\omega)$, we use $U \subseteq^{*} V$ to denote $U\backslash V \in {\rm Fin}$.

\section{Equivalence of functions and sequences}

 Let $f,g$ be two functions on $\mathbb{R}$ and $\mathbb{D} \subseteq \mathbb{R}$, we say $f$ and $g$ are \emph{equivalent} on  $\mathbb{D}$ if there is a constant $C \ge 1$ such that $g(x)/C \le f(x) \le Cg(x) $ for every $x \in \mathbb{D}$,  denote by $(f(x))_{x\in \mathbb{D}} \approx (g(x))_{x\in \mathbb{D}}$. Similarly, if $(u_n)_{n<\omega},(v_n)_{n<\omega}$ are two sequences on $\mathbb{R}$ and $U \in P(\omega)$, we say $(u_n)_{n<\omega},(v_n)_{n<\omega}$ are \emph{equivalent} on $U$ if there is a constant $C \ge 1$ such that $u_n/C \le v_n \le Cu_n $ for each $n \in U$, denoted $(u_n)_{n \in U} \approx (v_n)_{n\in U }$. In many cases, equivalence of functions and equivalence of sequences are related, we will give two propositions to illustrate this relationship, and we need the concept of  essentially increasing function firstly.

The definition of essentially increasing function can be found in  \cite{Ma}. Let $\mathbb{D} \subseteq \mathbb{R}$ and $f: \mathbb{D} \to \mathbb{R}$ be a function. We say $f$ is \emph{essentially increasing} on $\mathbb{D}$ if for some $C \ge 1, \forall x, y \in \mathbb{D}(x\le y \Rightarrow f(x) \le Cf(y))$. Equivalently, $f$ is  essentially increasing  on $\mathbb{D}$  if and only if there is an increasing function $g$ on $\mathbb{D}$ such that $(f(x))_{x\in \mathbb{D}} \approx (g(x))_{x\in \mathbb{D}}$. In fact, if $f$ is an essentially increasing function on $\mathbb{D}$, we set $g: \mathbb{D} \to \mathbb{R}, g(x)=\sup \{f(y):  y \le x\}$. It is easy to see that  $g$ satisfies the requirements.

\begin{proposition}\label{funseq1}
Let $f$ be an essentially increasing function on $[0,1]$ and $(x_n)_{n<\omega}$  a decreasing sequence on $[0,1]$ with $x_0=1$ and $\lim_{n \to \infty}{x_n}=0$. Assume there exists $\delta>0$ such that for each $n<\omega$, $f(x_{n+1})\ge \delta f(x_n)$. Let $g$ be a function on $[0,1]$, and there exists $K \ge 1$ such that for each $n<\omega$ and  $x \in [x_{n+1}, x_n]$, we have
\[
\min\{g(x_{n+1}) , g(x_n)\}/K \le g(x)\le K \max\{g(x_{n+1}), g(x_n)\}.
\]
If $(f(x_n))_{n<\omega} \approx (g(x_n))_{n<\omega}$, then $(f(x))_{x\in(0,1]} \approx (g(x))_{x \in (0,1]}$.
\end{proposition}

\begin{proof}

By assumption, we can select $C_1 \ge 1$ such that for each $n< \omega$ and $0 \le x \le y \le  1$,
\[
 f(x_n)/C_1 \le  g(x_n) \le C_1 f(x_n), ~ f(x) \le C_1f(y).
\]
Then for each $m \le n<\omega$,
 \[
g(x_m) \ge f(x_m)/C_1 \ge  f(x_n)/(C_1)^2 \ge g(x_n)/(C_1)^3,
\]
\[
g(x_{n+1}) \ge f(x_{n+1})/C_1 \ge \delta/C_1 \cdot f(x_n) \ge \delta/(C_1)^2 \cdot g(x_n).
\]

We show $g$ is essentially increasing  on $(0,1]$. Let $C_2=\max\{K^2(C_1)^3, K^2(C_1)^2/\delta\}$. If $0 \le x \le y \le 1$, then there are $m \le n <\omega$ such that  $x \in [x_{n+1}, x_n]$ and $y \in [x_{m+1}, x_m]$.

If $n>m$, then
\begin{align*}
g(x) & \le  K \max\{g(x_{n+1}), g(x_n)\}\\
 &\le K(C_1)^3 \min\{g(x_{m+1}), g(x_m)\} \le K^2(C_1)^3g(y)\le C_2g(y).
\end{align*}

If $n=m$,  then
\begin{align*}
    g(x) & \le  K \max\{g(x_{n+1}), g(x_n)\}\\
     & \le C_2/K \cdot \min\{g(x_{m+1}), g(x_m)\} \le C_2  g(y).
\end{align*}

For $x>0$, there is $m<\omega$ with $x \in [x_{m+1},x_m]$, therefore
\begin{align*}
g(x) & \ge g(x_{m+1})/C_2 \ge f(x_{m+1})/(C_1C_2)\\
 & \ge \delta f(x_m)/(C_1C_2) \ge \delta/((C_1)^2C_2) \cdot f(x),
\end{align*}
similarly, we have $g(x) \le (C_1)^2C_2/\delta \cdot f(x)$, therefore $(f(x))_{x\in(0,1]} \approx (g(x))_{x \in (0,1]}$.
\end{proof}

\begin{proposition}\label{funseq2}
Let $\alpha >1$ and $(x_n)_{n<\omega}$ be a sequence on $(0,1)$ with  $x_{n+1}=x_n^\alpha$ for each $n<\omega$.
Suppose $\varphi:[0,1] \to \mathbb{R}^+$ is essentially increasing on $[0,1]$  with $\varphi(x_0)>0$. If there is $\lambda>0$ such that $\varphi(x_{n+1}) \ge \lambda \varphi(x_n)$ for each $n<\omega$, then $(\varphi(x))_{x \in [0,1]} \approx (\varphi(x^\alpha))_{x \in [0,1]} $.


\end{proposition}
\begin{proof}
 Let $\psi:[0,1] \to \mathbb{R}^+$ be an increasing function with $(\varphi(x))_{x \in [0,1]} \approx (\psi(x))_{x\in [0,1]}$, then there is $K \ge 1$ such that $\varphi(x)/K  \le \psi(x)  \le K \varphi(x)$.  For each $n<\omega$,
\[
\psi(x_{n+1}) \ge \varphi(x_{n+1})/K \ge \lambda/K \cdot \varphi(x_n) \ge \lambda/K^2 \cdot \psi(x_n).
\]

Since $\varphi(x_0)>0$ and $\varphi(x)$ is essentially increasing on $[0,1]$ with $\varphi(x_{n+1}) \ge \lambda \varphi(x_n)$ for each $n<\omega$, we have $\varphi(x)>0$ for $x>0$, therefore $\psi(x)>0$ for $x>0$. Set $x_{-1}=1$, if $x>0$, then $x \in [x_{m+1},x_m]$ for some $m \in \omega\cup \{-1\}$.

If $m=-1$, then
\[
\psi(x^\alpha) \ge \psi(x_{1}) \ge \psi(x)\psi(x_1)/\psi(1).
\]

If $m \ge 0$, then
\begin{align*}
\psi(x^\alpha)&\ge \psi(x_{m+2}) \ge \lambda/K^2 \cdot  \psi (x_{m+1}) \\
&\ge \lambda^2/K^4 \cdot \psi(x_m) \ge \lambda^2/K^4 \cdot \psi(x).
\end{align*}

Let $C =\max\{((K^2\psi(1))/\psi(x_1), K^6/\lambda^2\}>0$, then for each $x \in [0,1]$, $\psi(x^\alpha) \ge K^2/C \cdot \psi(x) $, hence $\varphi(x^\alpha) \ge \varphi(x)/C$. Note that for each $x \in [0,1]$,
\[
\varphi(x^\alpha) \le K \psi(x^\alpha) \le K\psi(x) \le K^2 \varphi(x)\le C\varphi(x).
\]
Therefore, $(\varphi(x))_{x \in [0,1]} \approx (\varphi(x^\alpha))_{x \in [0,1]} $.
\end{proof}

Specially,  we have the following corollary.
\begin{corollary}\label{funseq3}
Let $\varphi:[0,1] \to \mathbb{R}^+$ be an essentially increasing function  with $\varphi(1/2)>0$. If there is $\lambda>0$ such that $\varphi(1/2^{2n}) \ge \lambda \varphi(1/2^n)$ for each $n<\omega$, then $(\varphi(x))_{x \in [0,1]} \approx (\varphi(x^2))_{x \in [0,1]} $.
\end{corollary}
\begin{proof}
Let $x_0=1/2$ and $x_{n+1}=x_n^2 $ for $n<\omega$, then $x_n=1/2^{2^n}$ for each $n<\omega$, therefore
\[
\varphi(x_{n+1})=\varphi(1/2^{2^{n+1}})\ge \lambda\varphi(1/2^{2^n})=\lambda\varphi(x_n).
\]
By Proposition \ref{funseq2}, we get the conclusion.
\end{proof}

\section{$\mathbf{E}_f$  equivalence relations}

Let $f: [0,1] \to \mathbb{R}^+$ be an arbitrary function, M\'{a}trai \cite{Ma} defined the relation $\mathbf{E}_f$ on $[0,1]^\omega$ by setting, for every $(x_n)_{n<\omega} , ~(y_n)_{n<\omega} \in [0,1]^\omega$,
\[
(x_n)\mathbf{E}_f (y_n) \Longleftrightarrow \sum_{n<\omega} f(|y_n-x_n|)< \infty.
\]
It is straightforward that $\mathbf{E}_f$ is a Borel relation whenever $f$ is Borel. If $f, g: [0,1] \to \mathbb{R}^+$ satisfying $(f(x))_{x \in [0,1]} \approx (g(x))_{x \in [0,1]}$, then $\mathbf{E}_f=\mathbf{E}_g$.

For $0<p<\infty$, it is well known that $\mathbb{R}^\omega/\ell_p$   and $[0,1]^\omega/\ell_p=\mathbf{E}_{{\rm Id}^p}$ are Borel equivalent, so we can assume that  $\mathbb{R}^\omega/\ell_p$ belongs to the $\mathbf{E}_f$'s.

The following proposition answers when $\mathbf{E}_f$ is an equivalence relation.

\begin{proposition}[M\'{a}trai \cite{Ma}, Proposition 2]\label{E-f}
Let $f: [0,1] \to \mathbb{R}^+$ be a bounded function. Then $\mathbf{E}_f$ is an equivalence relation if and only if the following conditions hold:
\begin{enumerate}
\item[${\rm(R_1)}$]  $f(0)=0$;

\item[${\rm (R_2)}$] there is a $C \ge 1$ such that for every $x,y \in [0,1]$ with $x+y \in [0,1]$,
\[
f(x+y) \le C(f(x)+f(y)),
\]
\[
f(x) \le C(f(x+y)+f(y)).
\]
\end{enumerate}
\end{proposition}

A nonreducibility result was obtained for a wider class of $\mathbf{E}_f$'s in \cite{Ma} as follows.

\begin{theorem} [M\'{a}trai \cite{Ma}, Theorem 18]\label{inredu}
Let $1\le \alpha <\infty$ and let $\varphi, \psi: [0,1] \to \mathbb{R}^+$ be continuous functions. Set $f(x)=x^\alpha \varphi(x), g(x)= x^\alpha \psi(x)$ for $x \in [0,1]$ and suppose that $f,g$ are bounded and $\mathbf{E}_f$ and $\mathbf{E}_g$ are equivalence relations. Suppose $\psi(x)>0(x>0)$, and
\begin{enumerate}
\item[${\rm (A_1)}$] there exists $\varepsilon>0, M<\omega$ such that for every $n>M$ and $x,y \in [0,1]$,
\[
\varphi(x) \le \varepsilon \varphi(y)\varphi(1/2^n)\Rightarrow x \le \frac{y}{2^{n+1}};
\]
\item[${\rm (A_2)}$] $\lim_{n\to \infty}\psi(1/2^n)/\varphi(1/2^n)=0$.
\end{enumerate}
Then $\mathbf{E}_g \nleq_B \mathbf{E}_f$.
\end{theorem}

In fact, we may replace condition ${\rm (A_2)}$ in the theorem by

${\rm (A_2)'}$ $\liminf_{n\to \infty}\psi(1/2^n)/\varphi(1/2^n)=0$.

Actually, if we select a sequence  $(n_k)_{k<\omega}$ such that $\lim_{k\to \infty}\psi(1/2^{n_k})/\varphi(1/2^{n_k})=0$ and modify $Z_k$ in the proof of Theorem 18 of \cite{Ma} by $Z_k=\{i/2^{n_k}, 0\le i \le 2^{n_k}\}$,  we can check  the proof is still valid under condition ${\rm (A_2)'}$, the proof is almost word for word, a copy of the proof of  Theorem 18.   If there is no confusion, we may say a function satisfies condition ${\rm (A_2)'}$ in Theorem \ref{inredu}  instead of satisfying condition ${\rm (A_2)}$ in Theorem \ref{inredu}. In this paper, condition ${\rm (A_2)'}$ is the key to prove incomparability between equivalence relations.

To satisfy ${\rm (A_1)}$, we give the following lemma:
\begin{lemma}\label{A1}
Let $\alpha \ge 1$ and $\varphi:[0,1] \to \mathbb{R}^+$ be an essentially increasing function with $\varphi(1/2)>0$. Set $f(x)=x^\alpha\varphi(x)$ for $x\in [0,1]$. If there exists $\delta >0$ such that for each $n<\omega$, $\varphi(1/2^{2n})\ge \delta \varphi(1/2^n)$, then $\mathbf{E}_f$ is an equivalence relation and condition ${\rm (A_1)}$ in Theorem \ref{inredu} holds.
\end{lemma}
\begin{proof}
Let $\psi:[0,1] \to \mathbb{R}^+$ be an increasing function with $(\varphi(x))_{x \in [0,1]} \approx (\psi(x))_{x\in [0,1]}$, then there is $K \ge 1$ such that $\varphi(x)/K  \le \psi(x)  \le K \varphi(x)$ for $x \in [0,1]$. For each $n<\omega$,
\[
\psi(1/2^{2n}) \ge \delta/K^2 \cdot \psi(1/2^n).
\]
Set $g(x)=x^\alpha\psi(x)$ for $x\in [0,1]$, then $\mathbf{E}_f=\mathbf{E}_g$. Since $\varphi(1/2)>0$ and $\psi$ is  increasing, we have $\psi(x)>0$ for $x>0$. To show $\mathbf{E}_g$ is an equivalence relation, by Proposition \ref{E-f}, we need only to check ${\rm (R_1)}$ and ${\rm (R_2)}$.

For ${\rm (R_1)}$, $g(0)=0$ is trivial.

For ${\rm (R_2)}$, let $x, y \in [0,1]$ with $x+y \in [0,1]$. Without loss of generality, we can assume that $x \ge y > 0$. Since $g(x)=x^\alpha \psi(x)$ is increasing, we have
\[
g(x) \le g(x+y) \le (g(x+y)+g(y)).
\]

If $x>1/8$, then
\[
g(x+y) \le g(1) \le \frac{g(1)}{g(1/8)}g(x)\le \frac{g(1)}{g(1/8)}(g(x)+g(y)).
\]

If $x\le 1/8$, let $x \in (1/2^{n+1},1/2^n]$ for some $n \ge 3$. Then
\[
g(x+y) \le g(2x) \le g(1/2^{n-1})=\frac{1}{2^{(n-1)\alpha}}\psi(1/2^{n-1}),
\]
and
\[
g(x)\ge g(1/2^{n+1}) \ge  \frac{1}{2^{(n+1)\alpha}}\psi(1/2^{2n-2})\ge \frac{\delta}{2^{(n+1)\alpha}\cdot K^2 }\psi(1/2^{n-1}).
\]
Therefore
\[
g(x+y) \le \frac{4^\alpha K^2}{\delta}(g(x)+g(y)).
\]

Let $C =\max\{g(1)/g(1/8), (4^\alpha K^2)/\delta\}$,
then $g(x+y) \le C(g(x)+g(y))$.
Therefore ${\rm (R_2)}$ holds and  $\mathbf{E}_g$ is an equivalence relation, so  $\mathbf{E}_f$ is
an equivalence relation as well.

For ${\rm (A_1)}$, fix a  $0<\varepsilon<\min \{1/(2\psi(1)), \delta^2/(2K^4\psi(1))\}$.
For $x,y \in [0,1]$ and $n\ge 2$, we  show if $x>y/2^{n+1}$, then
\[
\psi(x)> \varepsilon \psi(y)\psi(1/2^n).
\]

If $y=0$, then $\varepsilon \psi(0)\psi(1/2^n) \le \varepsilon \psi(0)\psi(1)\le \psi(0)/2<\psi(x)$.

If $y>0$, let $y \in (1/2^{m+1},1/2^m]$ for some $m \in \omega$, then $x>1/2^{m+n+2}$. Without loss of generality, we can assume $m \ge n$, then
\[
\psi(1/2^m)\psi(1/2^n) \le  K^2/\delta \cdot\psi(1/2^{2m}) \psi(1/2^n) \le K^2/\delta \cdot\psi(1/2^{m+n})\psi(1).
\]
Therefore
\begin{align*}
\varepsilon \psi(y)\psi(1/2^n) & \le \varepsilon \psi(1/2^m)\psi(1/2^n)
\le (\varepsilon K^2)/\delta \cdot \psi(1)\psi(1/2^{m+n}) \\
& \le (\varepsilon K^4)/\delta^2 \cdot \psi(1)\psi(1/2^{2m+2n}) \le (\varepsilon K^4)/\delta^2 \cdot \psi(1)\psi(1/2^{m+n+2})\\
&<\psi(x).
\end{align*}
So $\psi$ satisfies condition ${\rm (A_1)}$ in Theorem \ref{inredu}.
Let $\varepsilon'=\varepsilon/K^3$, for $x, y \in [0,1]$ and $n\ge 2$, if $\varphi(x) \le \varepsilon' \varphi(y)\varphi(1/2^n)$, then $\psi(x)\le \varepsilon \psi(y)\psi(1/2^n)$, hence $x \le y/2^{n+1}$, therefore $\varphi$ satisfies condition ${\rm (A_1)}$ in Theorem \ref{inredu}.
\end{proof}

The next theorem is a reducibility result for $\mathbf{E}_f$'s given in \cite{Ma} . The
original version in the  Real Analysis Exchange  contains  an error,  and the  revised version in his homepage(\url{http://www.renyi.hu/~matrait/}) has corrected this error by adding condition  \eqref{redu03}.
\begin{theorem}[M\'{a}trai \cite{Ma}, Theorem 16]\label{redu}
Let $f ,g : [0,1] \to \mathbb{R}^+$  be continuous essentially increasing functions such that $\mathbf{E}_f $ and $\mathbf{E}_g$ are equivalence relations. Suppose there exists a function $\kappa :\{1/2^i: i< \omega \} \to [0,1]$ satisfying the recursion
\begin{equation}\label{redu01}
f(1/2^n)= \sum_{i=0}^n g(\kappa (1/2^i)/2^{n-i})\quad(n<\omega)
\end{equation}
such that for some $L \ge 1$,
\begin{equation}\label{redu02}
\sum_{i=n}^\infty g(\kappa (1/2^i)) \le L \sum_{i=0}^n g(\kappa (1/2^i)/2^{n-i})\quad(n<\omega)
\end{equation}
and
\begin{equation}\label{redu03}
\kappa(1/2^n) \le L \cdot \max\{ \kappa (1/2^i)/2^{n-i}: i<n \}\quad(n< \omega).
\end{equation}
Then $\mathbf{E}_f \le_B \mathbf{E}_g$.
\end{theorem}

Mostly, we focus on equivalence relations $\mathbf{E}_f, \mathbf{E}_g$ where $f(x)=x^\alpha \varphi(x)$ and $g(x)=x^\alpha \psi(x)$ for $x \in [0,1]$.
Set $\kappa(1/2^n)=\mu(n)/2^{n}$, then $0 \le \mu(n) \le 2^n$ for each $n<\omega$ and  \eqref{redu01}, \eqref{redu02}, \eqref{redu03} turn to:
\begin{equation}\label{redu10}
    \varphi(1/2^n)=\sum_{i=0}^n \mu(i)^\alpha \psi(\mu(i)/2^n)\quad(n<\omega),
\end{equation}
\begin{equation}\label{redu12}
\sum_{i=0}^{\infty}\frac{\mu(n+i)^{\alpha}}{2^{i\alpha}}\psi\Big(\frac{\mu(n+i)}{2^{n+i}}\Big)  \le L  \sum_{i=0}^{n}\mu(i)^{\alpha}\psi\Big(\frac{\mu(i)}{2^{n}}\Big)\quad(n<\omega),
\end{equation}
\begin{equation}\label{redu13}
\mu(n) \le L \cdot \max_{i<n} \mu(i)\quad(n<\omega).
\end{equation}
Then we have the following easy lemma. The idea of  this lemma and Proposition \ref{apphold} come from  Corollary 29 of  \cite{Ma}.

\begin{lemma}\label{redu1}
Let $1 \le \alpha  <\infty$ and $\varphi , \psi: [0,1] \to \mathbb{R}^+$,  $\varphi$ is essentially increasing on $[0,1]$ and there is $\delta>0$ such that for each $n<\omega$, $\varphi(1/2^{n+1})\ge \delta \varphi(1/2^n)$. Set $f(x)=x^\alpha \varphi(x)$ and $g(x)=x^\alpha \psi(x)$ for $x\in [0,1]$, suppose $f, g$ are continuous essentially increasing  on $[0,1]$ and $\mathbf{E}_f$ and $\mathbf{E}_g$ are equivalence relations. Assume  there is $\mu: \omega \to \mathbb{R}^+$ satisfying  $\mu(n) \le 2^{n}$ for each $n<\omega$ and there is $L \ge 1$ such that for each $n<\omega$,  \eqref{redu12} and \eqref{redu13} hold. If
\begin{equation}\label{redu11}
        \Big(\varphi(1/2^n)\Big)_{n<\omega}\approx\Big(\sum_{i=0}^n \mu(i)^\alpha \psi(\mu(i)/2^n)\Big)_{n<\omega},
\end{equation}
 then $\mathbf{E}_f \le_B \mathbf{E}_g$.
\end{lemma}
\begin{proof}
Let $\theta: [0,1] \to \mathbb{R}^+$ with $\theta(0)=0$ and for each $n<\omega$, define $\theta(1/2^n)$ by
\[
\theta(1/2^n)= \sum_{i=0}^n \mu(i)^\alpha \psi(\mu(i)/2^n),
\]
 then extend $\theta$ to be a continuous function on $(0,1]$ which is affine on  $[1/2^{n+1}, 1/2^n]$ for $n<\omega$, therefore  $(\varphi(1/2^n))_{n<\omega} \approx (\theta(1/2^n))_{n<\omega}$, by Proposition \ref{funseq1},  we have $(\varphi(x))_{x>0}\approx (\theta(x))_{x>0}$.

  Set $h(x)=x^\alpha \theta(x)$ for  $x\in [0,1]$, then $(f(x))_{x\in(0,1]} \approx (h(x))_{x \in (0,1]}$, since $h(0)=f(0)=0$, we have  $(f(x))_{x \in [0,1]} \approx (h(x))_{x \in [0,1]}$, therefore $h$ is an essentially increasing function on $[0,1]$ and $\mathbf{E}_h$ is an equivalence relation with $\mathbf{E}_f=\mathbf{E}_h$.

  Since $\varphi$ is essentially increasing on $[0,1]$, we know $\varphi$ is bounded, therefore $\theta$ is bounded, so $h$ is continuous at zero. Since $\theta$ is continuous on $(0,1]$, we know that $h$ is continuous on $[0,1]$, therefore  $h$ is an  essentially increasing continuous function on $[0,1]$. 
By  Theorem \ref{redu}, we know  $\mathbf{E}_h \le_B \mathbf{E}_g$, therefore $\mathbf{E}_f \le_B \mathbf{E}_g$.
\end{proof}

Now let us take a further research about when condition \eqref{redu11} can be satisfied.

\begin{proposition}\label{apphold}
Let $\mu: \omega \to \mathbb{R}^+$ with $ \mu(n) \le 2^n$ for each $n<\omega$ and  $\varphi, \psi: [0,1] \to \mathbb{R}^+$, if
\begin{equation}\label{apphold1}
\Big(\varphi(1/2^n)\Big)_{n <\omega} \approx \Big( \psi(1/2^n)\sum_{i=0}^n \mu(i)^\alpha\Big)_{n <\omega}
\end{equation}
and there is $K \ge 1$ such that
\begin{equation}\label{apphold2}
\psi(1/2^n)/K \le \psi(\mu(i)/2^n) \le K\psi(1/2^n)\quad(0 \le i \le n <\omega, \mu(i)\neq 0),
\end{equation}
then \eqref{redu11} in Lemma \ref{redu1} holds.
\end{proposition}
\begin{proof}
Let $C \ge 1$ witnesses for each $n<\omega$,
\[
\varphi(1/2^n)/C \le \psi(1/2^n)\sum_{i=0}^n \mu(i)^\alpha \le C\varphi(1/2^n) .
\]
By \eqref{apphold2}, we have
\begin{align*}
\frac{1}{KC}\varphi(1/2^n) &\le \frac{1}{K}\psi(1/2^n)
\sum_{i=0}^n \mu(i)^\alpha
\le \sum_{i=0}^n \mu(i)^\alpha \psi(\mu(i)/2^n)\\
 &\le K\psi(1/2^n)\sum_{i=0}^n \mu(i)^\alpha \le KC\varphi(1/2^n).
\end{align*}
so \eqref{redu11} holds.
\end{proof}

Since $\mu$ may not satisfy $0\le \mu(n) \le 2^n$ for finitely many $n$, we need  the following proposition.

\begin{proposition}\label{changemu}
Let $\mu: \omega \to \mathbb{R}$ and $\varphi , \psi: [0,1] \to \mathbb{R}^+$ with $\varphi(x), \psi(x)>0$ for $x>0$, if there exist $n_0<\omega$ such that for $n \ge n_0, ~0 \le \mu(n)\le 2^n $ and
\begin{equation}\label{changemu1}
\Big(\varphi(1/2^n)\Big)_{n \ge n_0} \approx \Big( \psi(1/2^n)\sum_{i=0}^n \mu(i)^\alpha\Big)_{n \ge n_0}.
\end{equation}
Suppose there is  $K \ge 1$ such that
\begin{equation}\label{changemu2}
\psi(1/2^n)/K \le \psi(\mu(i)/2^n) \le K\psi(1/2^n)\quad(n_0 \le i \le n <\omega, \mu(i)\neq 0).
\end{equation}
Let $\nu: \omega \to \mathbb{R}^+$ with
\begin{equation*}
\nu(n)=
\begin{cases}
1, & n < n_0,\\
\mu(n), & n\ge n_0.
\end{cases}
\end{equation*}
Then $0 \le \nu(n)\le 2^n$ for each $n<\omega$  and  $\nu, \varphi, \psi$ satisfy  \eqref{apphold1} and \eqref{apphold2} in Proposition \ref{apphold}.
\end{proposition}
\begin{proof}
From the definition of the $\nu$, we know $0 \le \nu(n)\le 2^n$ for each $n<\omega$.

If $n < n_0$, since $\nu(n), \psi(1/2^n), \varphi(1/2^n)>0$, there is $K_1 \ge 1$ such that for $n <n_0$,
\begin{equation*}
\varphi(1/2^n)/K_1 \le \psi(1/2^n)\sum_{i=0}^n\nu(i)^\alpha \le K_1 \varphi(1/2^n).
\end{equation*}

If $n\ge n_0$, by \eqref{changemu1}, there is $K_2 \ge 1$ such that
\[
\varphi(1/2^n)/K_2 \le \psi(1/2^n)\sum_{i=0}^n \mu(i)^\alpha \le K_2\varphi(1/2^n),
\]
 therefore $\sum_{i=0}^{n_0}\mu(i)^\alpha>0$,  we can select $K_3 \ge 1$ such that
\[
\sum_{i=0}^{n_0}\mu(i)^\alpha/K_3 \le \sum_{i=0}^{n_0}\nu(i)^\alpha \le K_3\sum_{i=0}^{n_0}\mu(i)^\alpha.
\]
Hence for each $n \ge n_0$, we have

\begin{align*}
\psi(1/2^n)\sum_{i=0}^n \nu(i)^\alpha = & \psi(1/2^n)\Big(\sum_{i=0}^{n_0} \nu(i)^\alpha+\sum_{i=n_0+1}^n \nu(i)^\alpha\Big)\\
& \le \psi(1/2^n)\Big(K_3\sum_{i=0}^{n_0} \mu(i)^\alpha+\sum_{i=n_0+1}^n \mu(i)^\alpha\Big) \le K_3\psi(1/2^n)\sum_{i=0}^{n} \mu(i)^\alpha,
\end{align*}
similarly, we can get $\psi(1/2^n)\sum_{i=0}^n \nu(i)^\alpha \ge (\psi(1/2^n)\sum_{i=0}^{n} \mu(i)^\alpha)/K_3$,  therefore
 \[
 \varphi(1/2^n)/K_2K_3 \le \psi(1/2^n)\sum_{i=0}^n \nu(i)^\alpha \le K_2K_3 \varphi(1/2^n).
 \]

 Let $C=\max\{K_1, K_2K_3\}$, then for each $n<\omega$,
\[
\varphi(1/2^n)/C \le  \psi(1/2^n)\sum_{i=0}^n \nu(i)^\alpha \le C \varphi(1/2^n),
\]
so $\nu, \varphi, \psi$ satisfy condition \eqref{apphold1} in Proposition \ref{apphold}. It is easy to see that  $\nu, \varphi, \psi$ satisfy condition \eqref{apphold2} in Proposition \ref{apphold}.
\end{proof}

To satisfy condition \eqref{changemu2} in Proposition \ref{changemu}, we give the following proposition.

\begin{proposition}\label{changemu2hold}
Let  $\mu: \omega \to \mathbb{R}, \psi: [0,1] \to \mathbb{R}^+$ be functions. Suppose $\psi$ is  essentially increasing on $[0,1]$ with $\psi(1/2)>0$ and there is  $\lambda>0$ such that
\begin{equation}\label{changemu2hold1}
\psi(1/2^{2n }) \ge \lambda \psi(1/2^n)\quad(n<\omega).
\end{equation}
Assume there are $0 \le \varepsilon <1 ,~n_0<\omega $ with
\begin{equation}\label{changemu2hold2}
 2^{-\varepsilon n }\le \mu(i) \le 2^{\varepsilon n}\quad(n_0\le i \le n <\omega, \mu(i)\neq 0),
\end{equation}
 then  there is $K\ge 1$ such that
\[
\psi(1/2^n)/K \le \psi(\mu(i)/2^n) \le K\psi(1/2^n)\quad(n_0 \le i \le n <\omega, \mu(i)\neq 0).
\]
\end{proposition}
\begin{proof}
By Corollary \ref{funseq3},  there is $0<\delta<1$ such that $\psi(x)/\delta \ge \psi(x^2) \ge \delta \psi(x)$ for $x \in [0,1]$. Let $m<\omega$ satisfying $1- \varepsilon \in [1/2^{m+1}, 1/2^m]$. If $\mu(i)\neq 0$ with $ n_0\le i \le n <\omega$, then
\[
\frac{1}{2^{2n}}\le\frac{1}{2^{(1+\varepsilon) n }} \le \frac{\mu(i)}{2^n}\le \frac{1}{2^{(1-\varepsilon) n }}\le \frac{1}{2^{n/2^{m+1}}}.
\]
Since $\psi$ is essentially increasing, there is $C \ge 1$ such that $\psi(x) \le C\psi(y)$ for $0 \le x \le y \le 1$.
Therefore for $n_0\le i \le n <\omega$ with $ \mu(i)\neq 0$, we have
\[
\psi(\mu(i)/2^n)\ge \psi(1/2^{2n})/C \ge \delta/C  \cdot \psi(1/2^n)
\]
and
\[
\psi(\mu(i)/2^n) \le C\psi(1/2^{n/2^{m+1}}) \le C/\delta \cdot \psi(1/2^{n/2^{m}}) \le \cdots \le C/\delta^{m+1} \cdot \psi(1/2^n).
\]
Set $K=C/\delta^{m+1}$, then we get  the conclusion.
\end{proof}

\begin{corollary}\label{changemu2holdrem}
Let  $\mu: \omega \to \mathbb{R}, \psi: [0,1] \to \mathbb{R}^+$ be functions. Suppose $\psi$ is  essentially increasing on $[0,1]$ with $\psi(1/2)>0$ and there is  $\lambda>0$ such that
\[
\psi(1/2^{2n }) \ge \lambda \psi(1/2^n)\quad(n<\omega).
\]
If there are $0\le \varepsilon<1, n_0<\omega, M \ge 1$ such that
 \[
1/M \le \mu(i) \le 2^{\varepsilon n}\quad(n_0\le i \le n <\omega, \mu(i)\neq 0),
 \]
then there is $K \ge 1$ and $n_0 \le n_1<\omega$ such that
\[
\psi(1/2^n)/K \le \psi(\mu(i)/2^n) \le K\psi(1/2^n)\quad(n_1 \le i \le n <\omega, \mu(i)\neq 0).
\]
\end{corollary}

Let $1 \le \alpha <\beta<\infty$ and $\varphi: [0,1] \to \mathbb{R}^+$, set $f(x)=x^\alpha\varphi(x)$ and $g(x)=x^\beta$ for $x \in [0,1]$, we give the following lemma to show the Borel reduction between $\mathbf{E}_f$ and $\mathbf{E}_g$.

\begin{lemma}\label{lb}
Let $1\le \alpha <\infty$ and $\varphi: [0,1] \to \mathbb{R}^+$ be continuous essentially increasing with $\varphi(x)>0$ for $x>0$ and $\lim_{n\to \infty }\varphi(1/2^n)/\varphi(1/2^{n+1})=1$. Set $f(x)=x^\alpha\varphi(x)$ for $x \in [0,1]$, suppose $\mathbf{E}_f$ is an equivalence relation. Then for each $\alpha<\beta<\infty$, we have $\mathbf{E}_f \le_B \mathbf{E}_{\rm Id^{\beta}}$.
\end{lemma}
\begin{proof}
Let $\lambda=2^{\alpha-\beta}$.
Fix a $\varepsilon<\min\{1, 1/\lambda-1\}$, since $\lim_{n\to \infty }\varphi(1/2^n)/\varphi(1/2^{n+1})=1$, there is $n_0<\omega$ such that for $n\ge n_0$, we have
$1-\varepsilon \le \varphi(1/2^n)/\varphi(1/2^{n+1}) \le 1+\varepsilon$. Denote $M=\max\{\varphi(x):  x\in [0,1]\}$, then $M>0$.

Let $\psi: [0,1] \to \mathbb{R}^+$ defined by
\begin{equation*}
\psi(x)=
\begin{cases}
\varphi(1/2^{n_0})/M, & x \ge 1/2^{n_0} ,\\
\varphi(x)/M, & x < 1/2^{n_0}.
\end{cases}
\end{equation*}
It is easy to check that $(\varphi(x))_{x\in [0,1]} \approx (\psi(x))_{x\in [0,1]}$ and $\psi$ is continuous essentially increasing. Let $C \ge 1$ witness that $\psi(x) \le C\psi(y)$ for $x \le y$. Set $g(x)=x^\alpha\psi(x)$ for $x\in [0,1]$, then  $\mathbf{E}_f=\mathbf{E}_g$. To show $\mathbf{E}_g \le_B \mathbf{E}_{\rm Id^{\beta}}$, by Theorem \ref{redu}, we only need to find a function $\kappa :\{1/2^i: i<\omega\} \to [0,1]$ and $L \ge 1$ such that for each $n<\omega$,
\begin{enumerate}
  \item [(i)] $g(1/2^n)=\sum_{i=0}^n (\kappa(1/2^i)/2^{n-i})^\beta$;
  \item [(ii)] $\sum_{i=n}^\infty \kappa(1/2^i)^\beta \le L\sum_{i=0}^n (\kappa(1/2^i)/2^{n-i})^\beta=Lg(1/2^n)$;
  \item [(iii)] $\kappa(1/2^n) \le L \cdot \max \{\kappa(1/2^i)/2^{n-i}: i<n\}$.
\end{enumerate}

To satisfy (i), we  have to define $\kappa(1)^\beta=g(1)\le 1$ and for $n>0$,
\[
\kappa(1/2^n)^\beta=g(1/2^n)-g(1/2^{n-1})/2^\beta=(\psi(1/2^n)-\lambda\psi(1/2^{n-1}))/2^{n\alpha}.
\]
For $0<n \le n_0$,
\[
\kappa(1/2^n)^\beta=(1-\lambda)/2^{n\alpha}\cdot \varphi(1/2^{n_0})/M\in(0,1),
\]
for $n>n_0$,
\begin{align*}
\kappa(1/2^n)^\beta&=(1-\lambda\cdot \psi(1/2^{n-1})/\psi(1/2^n))\cdot \psi(1/2^n)/2^{n\alpha}\\
&\in (1-\lambda(1+\varepsilon), 1-\lambda(1-\varepsilon)) \cdot \psi(1/2^n)/2^{n\alpha} \subseteq (0,1).
\end{align*}
Therefore $\kappa(1/2^n)$ is well defined.

From the definition of $\kappa$, we have $\kappa(1/2^n)^\beta \le g(1/2^n)=\psi(1/2^n)/2^{n\alpha}$ for each $n<\omega$, therefore
\[
\sum_{i=n}^\infty \kappa(1/2^i)^\beta \le \sum_{i=n}^\infty \psi(1/2^i)/2^{i\alpha} \le \sum_{i=n}^\infty C\psi(1/2^n)/2^{i\alpha}=C/(1-2^{-\alpha})\cdot g(1/2^n).
\]

For (iii), if $0<n \le n_0+1$, we can find $L_1 \ge 1$ easily such that $\kappa(1/2^n) \le L_1 \cdot \kappa(1/2^{n-1})/2$; if  $n \ge n_0+2$, then
\begin{align*}
\kappa(1/2^n)^\beta&=(1-\lambda\psi(1/2^{n-1})/\psi(1/2^n)) \cdot \psi(1/2^n)/2^{n\alpha}\\
  &\le  (1-\lambda(1-\varepsilon)) \cdot \psi(1/2^n)/2^{n\alpha} \\
  &\le (1-\lambda(1-\varepsilon))/(1-\varepsilon) \cdot \psi(1/2^{n-1})/2^{n\alpha}\\
  &\le (1-\lambda(1-\varepsilon))/((1-\varepsilon)(1-\lambda(1+\varepsilon)))\\
   & \cdot (1-\lambda \psi(1/2^{n-2})/\psi(1/2^{n-1}))\psi(1/2^{n-1})/2^{n\alpha}\\
   &=(1-\lambda(1-\varepsilon))/(2(1-\varepsilon)(1-\lambda(1+\varepsilon))) \cdot \kappa(1/2^{n-1})^\beta.
\end{align*}

 Let $L=\max\{C/(1-2^{-\alpha}), L_1,  2[(1-\lambda(1-\varepsilon))/(2(1-\varepsilon)(1-\lambda(1+\varepsilon)))]^{1/\beta}\}$, then (ii) and (iii) are satisfied.
\end{proof}

\section{Embeddings of $P(\omega)/{\rm Fin}$ into $[\ell_\alpha, \ell_\beta]$}

In this section we will establish our main theorem. First we define a sequence on $(0,1]$ which has many nice properties.

\begin{lemma}\label{un}
Let $(\delta_m)_{m<\omega}$ be a sequence on $(0,1)$ satisfying $0<\inf_m \delta_m \le \sup_m \delta_m <1$ and $(k_m)_{m<\omega}$  an increasing sequence on $\omega\backslash \{0,1\}$ satisfying $k_{m+1} \ge 2k_m$ for each $m<\omega$.  Define $u_n$  by $u_0=1$ and for $n>0$,
\begin{equation*}
u_n=
\begin{cases}
\delta_m u_{n-1}, & n=k_m,\\
u_{n-1}, & {\rm otherwise}.
\end{cases}
\end{equation*}
Denote $\delta=\inf_m \delta_m$ and $\Delta=\sup_m \delta_m$. Then:
\begin{enumerate}
\item[{\rm(i)}] For each $m<\omega$, $u_{k_m}=\prod_{0 \le i \le m } \delta_i $, so if $k_m \le n <k_{m+1}$, then $u_n= u_{k_m}=\prod_{0 \le i \le m } \delta_i$;
\item[{\rm(ii)}] $u_{n} \ge u_{n+1}$ for each $n<\omega$ and $u_n \to 0$ as $n \to \infty$;
\item[{\rm(iii)}] $u_{2n} \ge \delta u_n$ for each $n<\omega$.
\end{enumerate}
\end{lemma}
\begin{proof}
To show (i), we do it by induction on $m$. If $m=0$, then $u_{k_0}=\delta_0$. Assume for $m = n , u_{k_n}= \prod_{0 \le i \le n } \delta_i$, then for $m=n+1$,
\[
u_{k_{n+1}}=\delta_{n+1} u_{k_{n+1}-1}=\delta_{n+1} u_{k_n}= \prod_{0 \le i \le n+1 } \delta_i,
\]
therefore (i) holds. Since $\Delta<1$, it is easy to see that (ii) is satisfied.

For (iii), if $2n < k_0$, then $u_n=u_{2n}=1$, otherwise there is $m<\omega$ such that
\[
k_m \le 2n < k_{m+1}.
\]
If $m=0$, by (i), $u_{2n}=\delta_0$ and $u_n=1$ or $u_n=\delta_0$, therefore $u_{2n} \ge \delta u_n$.  If $m>0$,
then $u_{2n}= u_{k_m}$ and
\[
k_{m-1} \le n < k_{m+1},
\]
if $n \ge k_m$, then $u_n= u_{k_m}$, otherwise $u_n= u_{k_{m-1}}$, since $u_{k_m}= \delta_{m} u_{k_{m-1}}$, therefore $u_{2n} \ge \delta u_n$ and (iii) is satisfied.
\end{proof}

If we select a subsequence $(k_{m_l})_{l<\omega}$ of $(k_m)_{m<\omega}$, and define  $u_n$  by $u_0=1$ and
\begin{equation*}
u_n=
\begin{cases}
\delta_{m_l} u_{n-1}, & n=k_{m_l},\\
u_{n-1}, & {\rm otherwise}.
\end{cases}
\end{equation*}
It is easy to see  that $(u_n)_{n<\omega}$ still satisfies (i)-(iii) in Lemma \ref{un}.

Now we are going to show that $P(\omega)/{\rm Fin}$ can be embedded into the set of Borel equivalence relations of  $\mathbf{E}_f$'s. For $l<\omega$, set $a_0=0$ and for $l>0$, let $a_l=a_{l-1}+1+ (l-1)a_{l-1}$. Set $I_l= [a_l, a_{l+1}) \cap \omega$. Let $(\delta_m)_{m<\omega}$ be a sequence on $(0,1)$ satisfying $0<\inf_m \delta_m \le \sup_m \delta_m <1$ and $(k_m)_{m<\omega}$  an increasing sequence on $\omega\backslash \{0,1\}$ satisfying $k_{m+1} \ge 2k_m$ for each $m<\omega$. Denote $\delta=\inf_m \delta_m$ and $\Delta=\sup_m \delta_m$. For every $U \in P(\omega)$, set $u_U(0)=1$ and for each $n>0$, define $u_U(n)$ by:
\begin{equation*}
u_U(n)=
\begin{cases}
\delta_m u_U(n-1), & n=k_m, m \in I_l, l \in U,\\
u_U(n-1), & \text{otherwise}.
\end{cases}
\end{equation*}
Then we have the following theorem:

\begin{theorem}\label{embedEf}
  Let $\alpha \ge 1 $, for each $U \in P(\omega)$ and $n<\omega$, let $u_U(n), I_n,  a_n,  k_n, \delta_n$  and $\delta, \Delta$ defined as above. Let $\varphi: [0,1] \to \mathbb{R}^{+}$ be  continuous essentially increasing with $\varphi(x)>0$ for $x>0$,
   and there is $\lambda>0$ such that $\varphi(x^2)\ge \lambda \varphi(x)$ for $x\in [0,1]$. For every $U \in P(\omega)$, let $\varphi_U: [0,1] \to \mathbb{R}^+$ be a continuous increasing function with  $\varphi_U(1/2^n)=u_U(n)$ for each $n<\omega$. Set $f_U(x)=x^\alpha\varphi(x) \varphi_U(x)$ for $x\in [0,1]$. Then $\mathbf{E}_{f_U}$ is an equivalence relation and for $U, ~V \in P(\omega)$, we have
\[
U\subseteq^{*}V\Longleftrightarrow \mathbf{E}_{f_U} \le_B \mathbf{E}_{f_V}.
\]
\end{theorem}

\begin{proof}
From Lemma \ref{un} (iii), we have $\varphi_U(1/2^{2n}) \ge \delta \varphi_U(1/2^n)$ for each $n<\omega $,
 therefore
\[
\varphi(1/2^{2n}) \varphi_U(1/2^{2n})\ge \delta \lambda \varphi(1/2^{n})\varphi_U(1/2^{n}),
\]
since $\varphi\varphi_U$ is an essentially increasing function on $[0,1]$
 and $\varphi(1/2)\varphi_U(1/2)>0$, by Lemma \ref{A1}, $\mathbf{E}_{f_U}$ is an equivalence relation.

 For $U, V \in P(\omega)$,  to show
  \[
 U\subseteq^{*}V  \Longleftrightarrow \mathbf{E}_{f_U} \le_B \mathbf{E}_{f_V}.
  \]
 We only need to consider three cases.

 \textbf{Case 1:} {\it For $U, ~V \in P(\omega)$, if $U\subseteq^{*}V$ and $V\subseteq^{*}U$, then $\mathbf{E}_{f_U} = \mathbf{E}_{f_V}$.}

 Since $\varphi(x)$ is essentially increasing on $[0,1]$, there is  $C\ge 1$ such that $\varphi(x) \le C\varphi(y)$ for $0 \le x\le y\le 1$. For $U, ~V \in P(\omega)$, if $U\subseteq^{*}V$ and $V\subseteq^{*}U$, then $U,~V$ differ only by a finite set, so $(u_U(n))_{n<\omega} \approx (u_V(n))_{n<\omega}$, therefore $(f_U(1/2^n))_{n<\omega} \approx (f_V(1/2^n))_{n<\omega}$.

 From Lemma \ref{un} (iii), for each $0<n<\omega$, we have:
\begin{align*}
f_U(1/2^{n+1})&=1/2^{(n+1)\alpha} \cdot \varphi(1/2^{n+1}) \varphi_U(1/2^{n+1})\\
&\ge 1/(2^{(n+1)\alpha}C)\cdot \varphi(1/2^{2n})\varphi_U(1/2^{2n})\\
& \ge \lambda\delta/(2^\alpha C) \cdot \varphi(1/2^{n}) \varphi_U(1/2^{n})/2^{n\alpha}\\
&= \lambda\delta/(2^\alpha C) \cdot f_U(1/2^{n}),
\end{align*}
for $n=0$,
\[
f_U(1/2)=1/2^{\alpha} \cdot \varphi(1/2)\ge \varphi(1/2)/(\varphi(1)2^{\alpha}) \cdot f_U(1).
\]

If $x \in [1/2^{n+1}, 1/2^n]$ for some $n<\omega$, then
\[
f_V(1/2^{n+1})/C \le f_V(x) \le Cf_V(1/2^n),
\]
by Proposition \ref{funseq1},  $(f_U(x))_{x\in(0,1]} \approx (f_V(x))_{x\in(0,1]}$. Since $\varphi\varphi_U, \varphi\varphi_V$ is an essentially increasing function on $[0,1]$, therefore $f_U(0)= f_V(0)=0$, so $(f_U(x))_{x\in [0,1]} \approx (f_V(x))_{x\in[0,1]}$ and $\mathbf{E}_{f_U} = \mathbf{E}_{f_V}$.

\textbf{Case 2:} {\it For $U, ~V \in P(\omega)$, if $U\subseteq^{*}V$ with $|V\backslash U|=\infty$, then $\mathbf{E}_{f_U} <_B \mathbf{E}_{f_V}$.}

Let $U'=U \cap V $,  then $ U\subseteq^{*}U',  U'\subseteq^{*} U$ and $|V\backslash U'|=
\infty$. From Case 1, we know that $\mathbf{E}_{f_U} = \mathbf{E}_{f_{U'}}$. Without loss of generality, we can assume $U \subseteq V$ with $|V\backslash U|=\infty$, then $1 \ge u_U(n)\ge u_V(n)>0$ for each $n<\omega$.

Define $\mu: \omega \to \mathbb{R}$ with $\mu(0)=1$ and for $n>0$,
\[
\mu(n)^\alpha=\frac{\varphi_U(1/2^n)}{\varphi_V(1/2^n)}-\frac{\varphi_U(1/2^{n-1})}{\varphi_V(1/2^{n-1})}=
\frac{u_U(n)}{u_V(n)} -\frac{u_U(n-1)}{u_V(n-1)}.
\]
We are going to give several claims to prove Case 2.

\textbf{Claim 1:} {\it  $\mu(n) \ge 0$ for each $n<\omega$ and there exists $n_0<\omega$ such that for $n \ge n_0$, if $\mu(n) \neq 0$, then
$1/\Delta- 1 \le \mu(n)^\alpha \le 2^{n/2}$.}

{\it \noindent Proof of Claim 1.} If $u_U(n)=u_U(n-1), u_V(n)=u_V(n-1)$ or $u_U(n)=\delta_m u_U(n-1), u_V(n)=\delta_m u_V(n-1)$ for some $m<\omega$, then $\mu(n)=0$.

If $u_U(n)=u_U(n-1), u_V(n)=\delta_m u_V(n-1)$ for some $m<\omega$, then  we have
\[
\mu(n)^\alpha \ge  (1/\Delta-1)\cdot (u_U(n-1)/ u_V(n-1)) \ge 1/\Delta -1>0
\]
and
\[
\mu(n)^\alpha \le (1/\delta_m-1)/ u_V(n-1)=
(1-\delta_m)/u_V(n).
\]
Since $u_V(n)=\delta_m u_V(n-1)$, there is $l \in V$ such that  $m \in I_l$ and $n=k_m$. Let $I_V= \bigcup_{l\in V} I_l$ and assume $m$ is the $p$-th number in $I_V$, then $p \le m+1$.  By induction on  $m$, we have $k_m \ge 2^{m+1}$ , so $2^p \le k_m=n$. From Lemma \ref{un} (ii), we have
\[
1/u_V(n) \le (1/\delta)^p =2^{p\log_2(1/\delta)} \le n^{\log_2(1/\delta)},
\]
therefore $\mu(n)^\alpha \le n^{\log_2(1/\delta)}(1-\delta)$.
So there exists $n_0<\omega$ such that for $n \ge n_0$, if $\mu(n)\neq 0 $, then
\[
1/\Delta -1  \le \mu(n)^\alpha \le 2^{n/2}.
\]
 \hfill $\Box$ Claim 1

\textbf{Claim 2:} {\it There exist $ K_0 \ge 1$ and $n_1\ge n_0$ such that for $n \ge n_1$, we have  $0 \le \mu(n)\le 2^n $ and  $\mu, \varphi\varphi_U, \varphi\varphi_V$ satisfy the following requirements:
\begin{enumerate}
\item[{\rm (i)}] $ (\varphi(1/2^n)\varphi_U(1/2^n))_{n \ge n_1}\approx (\varphi(1/2^n)\varphi_V(1/2^n)\sum_{i=0}^n \mu(i)^\alpha)_{n \ge n_1}$,

\item[{\rm(ii)}] $ \varphi(1/2^n)\varphi_V(1/2^n)/K_0  \le \varphi(\mu(i)/2^n)\varphi_V(\mu(i)/2^n) \le K_0 \cdot \varphi(1/2^n)\varphi_V(1/2^n) \quad \\(n_1 \le i \le n <\omega,~ \mu(i)\neq 0)$.
\end{enumerate}

}
{\it \noindent Proof of Claim 2.} For each $n<\omega$,
 \[
 \varphi(1/2^n)\varphi_V(1/2^n)\sum_{i=0}^n \mu(i)^\alpha = \varphi(1/2^n)\varphi_U(1/2^n),
 \]
 so (i) holds.

For (ii) , if $\mu(i)\neq 0$ with $n_0 \le i \le n <\omega$, then
\[
1/\Delta -1 \le \mu(i)^\alpha \le 2^{n/2}.
\]
For each $n<\omega$, we have
\[
\varphi(1/2^{2n}) \varphi_V(1/2^{2n})\ge \lambda \delta \varphi(1/2^{n})\varphi_V(1/2^{n}).
\]
 Since $\varphi \varphi_V$ is an essentially increasing  function with $\varphi(x) \varphi_V(x)>0$ for $x>0$, by Corollary \ref{changemu2holdrem}, there exists $K_0 \ge 1$ and  $n_1\ge n_0$ such that
(ii) is satisfied.

\hfill $\Box$ Claim 2

Let $\nu: \omega \to \mathbb{R}^+$ with
\begin{equation*}
\nu(n)=
\begin{cases}
1, & n < n_1,\\
\mu(n), & n\ge n_1.
\end{cases}
\end{equation*}
Then  $0 \le \nu(n) \le 2^n$ for each $n<\omega$ and $\nu(0)=1$.
By Proposition  \ref{changemu}, we know that there is $K \ge K_0$ such that
\[
   \Big(\varphi(1/2^n)\varphi_U(1/2^n)\Big)_{n<\omega}\approx\Big( \varphi(1/2^n)\varphi_V(1/2^n)\sum_{i=0}^n \nu(i)^\alpha \Big)_{n<\omega}
\]
and
\begin{align}\label{equ2}
\frac{1}{K}\varphi(1/2^n)\varphi_V(1/2^n) &\le \varphi(\nu(i)/2^n)\varphi_V(\nu(i)/2^n) \\
&\le K\varphi(1/2^n)\varphi_V(1/2^n)~~(0 \le i \le n <\omega, \nu(i)\neq 0).\notag
\end{align}
By Proposition \ref{apphold}, we have
\[
   \Big(\varphi(1/2^n)\varphi_U(1/2^n)\Big)_{n<\omega}\approx\Big( \sum_{i=0}^n \nu(i)^\alpha \varphi(\nu(i)/2^n)\varphi_V(\nu(i)/2^n) \Big)_{n<\omega}.
\]

In Case 1, we have shown that for each $n<\omega$,
\[f_U(1/2^{n+1})\ge \min\{\lambda\delta/(2^\alpha C), \varphi(1/2)/(\varphi(1)2^\alpha) \}f_U(1/2^n),
 \]
 therefore there is $\gamma>0$ such that
\[
\varphi(1/2^{n+1})\varphi_U(1/2^{n+1}) \ge \gamma \varphi(1/2^{n})\varphi_U(1/2^{n}).
\]
Since $f_U, f_V, \varphi\varphi_U, \varphi\varphi_V$ are continuous essentially increasing functions on $[0,1]$ and $\mathbf{E}_{f_U}, \mathbf{E}_{f_V}$ are equivalence relations. By Lemma \ref{redu1}, to show $ \mathbf{E}_{f_U}\le_B \mathbf{E}_{f_V}$, we only have to check  the  following claim.

 \textbf{Claim 3:} {\it There is $L \ge 1$ such that $\nu, \varphi\varphi_V$ satisfy the following requirements for each $n<\omega$:
\begin{enumerate}
\item[{\rm (i)}] $\sum_{i=0}^{\infty}\dfrac{\nu(n+i)^{\alpha}}{ 2^{i\alpha}} \varphi\Big(\dfrac{\nu(n+i)}{2^{n+i}}\Big)\varphi_V\Big(\dfrac{\nu(n+i)}{2^{n+i}}\Big)  \le L  \sum_{i=0}^{n}\nu(i)^{\alpha} \varphi\Big(\dfrac{\nu(i)}{2^{n}}\Big)\varphi_V\Big(\dfrac{\nu(i)}{2^{n}}\Big)$,
\item[{\rm(ii)}] $\nu(n) \le L \cdot \max_{i<n}\nu(i)$.
\end{enumerate}
}
{\it \noindent Proof of Claim 3.} Set $n_2=\max\{n_1, k_0\}$, for $n \ge n_2$, let $l \in \omega$ satisfying $k_{a_l} \le n <k_{a_{l+1}}$. If $l \in U$, or $l \in \omega\backslash (U\cup V)$, or $l \in V\backslash U$ with $n \ge k_{a_{l+1}-2}$, then we choose $l'\in V\backslash U$ be the minimal natural number such that $l'>l$. Set $n'=k_{a_{l'}+1}$,  then $n'\ge 2n $ and
\[
1/\Delta \cdot u_U(n)/u_V(n) \le u_U(n')/u_V(n')\le (1/\delta)^3\cdot u_U(n)/u_V(n).
\]
If $l  \in V\backslash U$ with $n < k_{a_{l+1}-2}$, then there is $p \in [a_l, a_{l+1}-2) \cap \omega $ such that $k_p \le n <k_{p+1} $. Since $I_l=[a_l, a_{l+1})\cap \omega$, we know that $p+2\in I_l$, let $n'=k_{p+2}$, then $n'\ge 2n $ and
\[1/\Delta \cdot u_U(n)/u_V(n) \le u_U(n')/u_V(n')\le (1/\delta)^2\cdot u_U(n)/u_V(n).
\]
 Therefore for $n \ge n_2$, we can select $m \in \omega$ with $m \ge n$ such that
\[
1/\Delta \cdot u_U(n)/u_V(n) \le u_U(n+m)/u_V(n+m)\le  (1/\delta)^3 \cdot u_U(n)/u_V(n).
\]

Denote the left part of equation (i) by $ EQU_{l}^n$, and the right part of equation (i) by $ EQU_{r}^n$ (don't contain $L$). Note that for $n>0$, we have
\[
u_U(n)/u_V(n) \le 1/\delta \cdot u_U(n-1)/u_V(n-1).
\]
For $n \ge n_1$, if $\mu(n)\neq 0 $, then
\[
1/\Delta -1  \le \mu(n)^\alpha=\nu(n)^\alpha \le 2^{n/2}.
\]

By \eqref{equ2}, we have
\[
EQU_{l}^n \le K\sum_{i=0}^{\infty}\frac{1}{2^{i\alpha}}\nu(n+i)^{\alpha}\varphi\Big(\frac{1}{2^{n+i}}\Big)\varphi_V\Big(\frac{1}{2^{n+i}}\Big)
\]
and
\[
EQU_r^n \ge \frac{1}{K}\varphi(1/2^n)\varphi_V(1/2^n)\sum_{i=0}^n \nu(i)^\alpha.
\]
Therefore for $n \ge n_2$, we have
\begin{align*}
EQU_{l}^n &\le K\sum_{i=0}^{\infty}\frac{1}{2^{i\alpha}}\mu(n+i)^{\alpha}\varphi\Big(\frac{1}{2^{n+i}}\Big)\varphi_V\Big(\frac{1}{2^{n+i}}\Big)\\
&\le KC\sum_{i=0}^{m}\mu(n+i)^{\alpha}\varphi\Big(\frac{1}{2^{n}}\Big)\varphi_V\Big(\frac{1}{2^{n}}\Big)
+KC\sum_{i=m+1}^{\infty}\frac{1}{2^{i\alpha}}2^{(n+i)/2}\varphi\Big(\frac{1}{2^{n}}\Big)\varphi_V\Big(\frac{1}{2^{n}}\Big)\\
& \le \Big(\frac{u_U(n+m)}{u_V(n+m)}- \frac{u_U(n-1)}{u_V(n-1)}+ \frac{2^{n/2} }{2^{m(\alpha-1/2)}(2^{(\alpha-1/2)}-1)}\Big)KC\varphi\Big(\frac{1}{2^{n}}\Big)\varphi_V\Big(\frac{1}{2^{n}}\Big)\\
& \le\Big( (1/\delta^3 -\delta)\frac{u_U(n)}{u_V(n)} +\frac{1}{2^{(\alpha-1/2)}-1}\Big)KC\varphi\Big(\frac{1}{2^{n}}\Big)\varphi_V\Big(\frac{1}{2^{n}}\Big),
\end{align*}
and
\[
EQU_r^n \ge \Big(\frac{u_U(n)}{u_V(n)} -\frac{u_U(n_1-1)}{u_V(n_1-1)}+n_1\Big)\varphi(1/2^n)\varphi_V(1/2^n)/K.
\]

Since $U \subseteq V$ and $|V\backslash U|=\infty $, we have $u_U(n)/u_V(n)\to \infty$ as $n \to \infty$,  therefore there is a natural number $n_3\ge n_2$ and a positive real number $L_1\ge 1$ such that for $n \ge n_3$,
\[
(1/\delta^3 -\delta)\frac{u_U(n)}{u_V(n)} +\frac{1}{2^{(\alpha-1/2)}-1}
\le L_1\Big(\frac{u_U(n)}{u_V(n)}-\frac{u_U(n_1-1)}{u_V(n_1-1)}+n_1\Big).
\]
For $n_2 \le n <n_3$, since $u_U(n)/u_V(n) -u_U(n_1-1)/u_V(n_1-1)+n_1$ and $(1/\delta^3 -1)u_U(n)/u_V(n) +1/2^{(\alpha-1/2)-1}$ are positive, there is $L_2\ge 1$ such that for $n_2 \le n <n_3$,
\[
(1/\delta^3 -\delta)\frac{u_U(n)}{u_V(n)} +\frac{1}{2^{(\alpha-1/2)}-1}\le L_2\Big(\frac{u_U(n)}{u_V(n)}
-\frac{u_U(n_1-1)}{u_V(n_1-1)}+n_1\Big).
\]

Let $L_3=\max\{L_1CK^2, ~L_2CK^2\}$, then for $n \ge n_2$,
\[
EQU_{l}^n \le L_3\cdot EQU_{r}^n.
\]

For $n<n_2$, we have
\begin{align*}
   EQU_{l}^n &=  \sum_{i=0}^{n_2-1}\frac{1}{2^{i\alpha}}\nu(n+i)^{\alpha}\varphi\Big(\frac{\nu(n+i)}{2^{n+i}}\Big)\varphi_V\Big(\frac{\nu(n+i)}{2^{n+i}}\Big)  \\
    &+\sum_{i=0}^{\infty}\frac{1}{2^{(n_2+i)\alpha}}\nu(n+n_2+i)^{\alpha}\varphi\Big(\frac{\nu(n+n_2+i)}{2^{n+n_2+i}}\Big)\varphi_V\Big(\frac{\nu(n+n_2+i)}{2^{n+n_2+i}}\Big)\\
    &\le \sum_{i=0}^{n_2-1}\frac{1}{2^{i\alpha}}\nu(n+i)^{\alpha}\varphi\Big(\frac{\nu(n+i)}{2^{n+i}}\Big)\varphi_V\Big(\frac{\nu(n+i)}{2^{n+i}}\Big)  \\
    &+\frac{L_3}{2^{n_2\alpha}}
    \sum_{i=0}^{n+n_2}\nu(i)^{\alpha}\varphi\Big(\frac{\nu(i)}{2^{n+n_2}}\Big)\varphi_V\Big(\frac{\nu(i)}{2^{n+n_2}}\Big).
\end{align*}
It is easy to see that the right part of the formula is bounded, so there exists  $M>0$ such that $EQU_{l}^n \le M\varphi(1/2^{n_2})\varphi_V(1/2^{n_2})$ for each $n<n_2$. Note that $\nu(0)=1$ and $\varphi(x)\le C\varphi(y)$ for $x \le y$,  therefore
\begin{align*}
EQU_{l}^n &\le  M\varphi(1/2^{n_2})\varphi_V(1/2^{n_2})\\
&\le MC\nu(0)^{\alpha}\varphi\Big(\frac{\nu(0)}{2^{n}}\Big)\varphi_V\Big(\frac{\nu(0)}{2^{n}}\Big)\\
 & \le MC\sum_{i=0}^{n}\nu(i)^{\alpha} \varphi\Big(\frac{\nu(i)}{2^{n}}\Big)\varphi_V\Big(\frac{\nu(i)}{2^{n}}\Big)=MC\cdot EQU_r^n.
\end{align*}

Let $L_0=\max \{L_3, MC\}$, then for each $n<\omega$,
\[
EQU_{l}^n \le L_0\cdot EQU_{r}^n.
\]

For (ii), if $v(n)=0$, then $\nu(n) \le \nu(0)=1$, if $\nu(n)\neq 0$, then there is $l \in V\backslash U$ and $m \in I_l$ with $n=k_m$. Set
 \[
 A_l= \Big\{l': l'<l, ~l' \in V\backslash U, ~ \forall m' \in I_{l'}(k_{m'} \ge n_1)\Big\}.
 \]
Without loss of generality, we can assume that $l$ is large enough such that the set $A_l\neq \emptyset$

Let $I_{V\backslash U}=\bigcup_{l \in V\backslash U}I_l$, choose $m'$ to be the largest number in $I_{V\backslash U}$ such that $m'<m$. More precisely,
if $m=a_l$, then let $l'$ be the largest number of $A_l$ and  $m'=a_{l'+1}-1$, if $m>a_l$, let  $m'=m-1\ge a_l$. Set $n'=k_{m'}$, then
 \begin{align*}
 \nu(n)^\alpha= \mu(n)^\alpha &= u_U(n-1)/(\delta_m u_V(n-1))-  u_U(n-1)/ u_V(n-1)\\
  &\le (1/\delta -1) u_U(n-1)/ u_V(n-1)=  (1/\delta -1) u_U(n')/ u_V(n')\\
  &=(1/\delta -1)u_U(n'-1)/ (\delta_{m'} u_V(n'-1))\\
  &\le 1/\delta \cdot (1/\delta -1) u_U(n'-1)/ u_V(n'-1)
 \end{align*}
 and
 \begin{align*}
 \nu(n')^\alpha =\mu(n')^\alpha &=u_U(n'-1)/(\delta_{m'} u_V(n'-1))-  u_U(n'-1)/ u_V(n'-1)\\
 &\ge(1/\Delta -1) u_U(n'-1)/ u_V(n'-1).
 \end{align*}
 Therefore $ \nu(n)^\alpha \le \Delta(1-\delta)/(\delta^2 (1-\Delta)) \cdot \nu(n')^\alpha$.  

 Set $L=\max\{L_0,(\Delta(1-\delta)/(\delta^2 (1-\Delta)))^{1/\alpha} \}$, then for each $n<\omega$, (i) and (ii) in Claim 3 hold.

\hfill $\Box$ Claim 3

By  Lemma \ref{redu1},  we have $ \mathbf{E}_{f_U}\le_B \mathbf{E}_{f_V}$.  $\mathbf{E}_{f_V} \nleq_B\mathbf{E}_{f_U}$ is an easy corollary of the following Case 3. Therefore we have $\mathbf{E}_{f_U} <_B \mathbf{E}_{f_V}$.

\textbf{Case 3:} {\it For $U, ~V \in P(\omega)$, if $|U\backslash V|=\infty$ and $|V\backslash U|=\infty$, then $\mathbf{E}_{f_U}$ and $\mathbf{E}_{f_V}$ are Borel incomparable.}

For $U \in P(\omega)$ and $n<\omega$, by Lemma \ref{un} (iii), we have
\[
\varphi(1/2^{2n}) \varphi_U(1/2^{2n})\ge \lambda \delta \varphi(1/2^{n})\varphi_U(1/2^{n}) .
\]
From Lemma \ref{A1}, we know that $\varphi\varphi_U, \varphi\varphi_V$ satisfies condition (${\rm A_1}$) in Theorem \ref{inredu}. Since $0<\delta =\inf_m \delta_m \le \Delta= \sup_m \delta_m<1$, there is a natural number $p\ge 1$ such that $\Delta^p \le  \delta$. By assumption,  $|U\backslash V|=\infty$ and $|V\backslash U|=\infty$,  we can choose two strictly increasing sequences $(u_l)_{l<\omega}$ and $(v_l)_{l<\omega}$ such that for each $l<\omega$,
\begin{enumerate}
  \item [\rm (i)] $u_l \in U\backslash V$ and $v_l \in V\backslash U$,
  \item [\rm (ii)] $u_l < v_l <u_{l+1}$ and $u_0 \ge p+1$.
\end{enumerate}

From (ii), we know that $u_l \ge p+1+2l$ and $v_l \ge p+2+2l$ for each $l<\omega$. Set $m_l=a_{u_l+1}-1, ~n_l=a_{v_l+1}-1$ for each $l<\omega$. Since  $a_{u_l+1}=a_{u_l}+1+u_l a_{u_l}$,   therefore for each $l<\omega$,
\[
\frac{u_U(k_{m_l})}{u_V(k_{m_l})}\le \frac{\Delta^{u_l a_{u_l}}}{\delta^{ a_{u_l}} } \le \frac{\Delta^{(p+1+2l)a_{u_l}}}{\delta^{a_{u_l}} } \le (\frac{\Delta^p}{\delta})^{a_{u_l}}\cdot \Delta^{2l} \le \Delta^{2l},
\]
similarly, we have $u_V(k_{n_l})/u_U(k_{n_l})\le \Delta^{2l}$ for each $l<\omega$. Let $x_l=1/2^{k_{m_l}}$ and $y_l=1/2^{k_{n_l}}$ for $l<\omega$, then
\[
\lim_{l\to \infty}\varphi_U(x_l)/ \varphi_V(x_l)=0, \lim_{l\to \infty}\varphi_V(y_l)/ \varphi_U(y_l)=0.
\]
So condition ${\rm (A_2)'}$ in Theorem \ref{inredu} holds, therefore, $\mathbf{E}_{f_U}$ and $\mathbf{E}_{f_V}$ are Borel incomparable.

Cases 1-3 show that $P(\omega)/{\rm Fin}$ can be embedded into $\mathbf{E}_{f}$'s and we complete the proof.
\end{proof}

\begin{lemma}\label{embedlb}
 Let $\alpha \ge 1 $,  for each $U \in P(\omega)$ and $n<\omega$, let $u_U(n), I_n, a_n, k_n, \delta_n$  and $\delta, \Delta$ defined as in Theorem \ref{embedEf}. Let $\varphi_U: [0,1] \to \mathbb{R}^+$ be  continuous increasing with  $\varphi_U(1/2^n)=u_U(n)$ for each $n<\omega$. Set $f_U(x)=x^\alpha\varphi_U(x)$ for $x\in [0,1]$. Then for each $\alpha< \beta <\infty$, we have
\[
\mathbb{R}^\omega/\ell_\alpha \le_B \mathbf{E}_{f_U} \le_B \mathbb{R}^\omega/\ell_\beta.
\]
\end{lemma}
\begin{proof}
From Theorem \ref{embedEf}, we have
\[
\mathbf{E}_{f_\emptyset} \le_B \mathbf{E}_{f_U} \le_B \mathbf{E}_{f_\omega}.
\]

If $U=\emptyset$, then $\varphi_U(1/2^n)=1$ for $n<\omega$, since $\varphi_U$ is continuous increasing in $[0,1]$, it follows $\varphi_\emptyset \equiv 1$, therefore $f_\emptyset(x)=x^\alpha$ for $x \in [0,1]$,  since $\mathbb{R}^\omega/\ell_\alpha$ and $\mathbf{E}_{\rm Id^\alpha}$  are Borel equivalent, therefore $\mathbb{R}^\omega/\ell_\alpha \le_B \mathbf{E}_{f_U}$.

If $U=\omega$, by the definition of $(u_U(n))_{n<\omega}$, we know that  $u_\omega(0)=1$ and for $n>0$,
\begin{equation*}
u_\omega(n)=
\begin{cases}
\delta_m u_\omega(n-1), & n=k_m,\\
u_\omega(n-1), & \text{otherwise}.
\end{cases}
\end{equation*}

Let $F_\omega: \mathbb{R}^+\to \mathbb{R}^+$ with $F_\omega(0)=u_\omega(0)$ and $F_\omega(k_m)=u_\omega(k_m)$ for each $m<\omega$, then extend $F_\omega$ to $ \mathbb{R}^+$ to be a continuous decreasing function which is affine on $[0, k_0]$ and $[k_m, k_{m+1}]$ for each $m<\omega$.

Let $\psi_\omega: [0,1] \to \mathbb{R}^+$ with $\psi_\omega(1/2^n)=F_\omega(n)$  for each $n<\omega$, and then extend $\psi_\omega$ to $[0,1]$ to be a continuous increasing function on $[0,1]$.

Let $x_0=1$ and for $n>0$, set $x_n=1/2^{k_{n-1}}$, then $\varphi_\omega(x_{n+1}) \ge \delta \varphi_\omega(x_{n})$ for each $n<\omega$. Since $\varphi_\omega(x_{n})=\psi_\omega(x_{n})$ for each $n<\omega$, by Proposition \ref{funseq1}, we have $(\varphi_\omega(x))_{x\in(0,1]}\approx(\psi_\omega(x))_{x\in(0,1]}$. Since $\varphi_\omega, \psi_\omega$ are continuous increasing functions on $[0,1]$ and they are equal at $x_n$ for each $n<\omega$, therefore $\varphi_\omega(0)=\psi_\omega(0)$ and $(\varphi_\omega(x))_{x\in[0,1]}\approx(\psi_\omega(x))_{x\in[0,1]}$. Let $ g_\omega(x)=x^\alpha\psi_\omega(x)$ for $x\in [0,1]$, then $\mathbf{E}_{f_\omega}=\mathbf{E}_{g_\omega}$.

For $n \in [k_m, k_{m+1})$,
\[
\psi_\omega(1/2^n)-\psi_\omega(1/2^{n+1})=F(n)-F(n+1)=(F(k_m)-F(k_{m+1}))/(k_{m+1}-k_m),
\]
therefore
\[
 \frac{ \psi_\omega(1/2^n)-\psi_\omega(1/2^{n+1})}{\psi_\omega(1/2^{n+1})} \le \frac{u_\omega(k_m)-u_\omega(k_{m+1})}{u_\omega(k_{m+1})(k_{m+1}-k_m)}=\frac{1/\delta_{m+1}-1}{k_{m+1}-k_m}\le \frac{1/\delta_{m+1}-1}{k_m},
\]
since $0<1/\delta_{m+1}-1 \le 1/\delta-1$ for each $m<\omega$ and $k_m \to \infty$ as $m \to \infty$, we have $\lim_{n\to \infty} (\psi_\omega(1/2^n)-\psi_\omega(1/2^{n+1}))/\psi_\omega(1/2^{n+1})=0$, therefore,
\[
\lim_{n\to \infty} \psi_\omega(1/2^n)/\psi_\omega(1/2^{n+1})=1,
 \]
 using Lemma \ref{lb} we get that $\mathbf{E}_{g_\omega} \le_B \mathbf{E}_{\rm Id^\beta}$, therefore $\mathbf{E}_{f_\omega} \le_B \mathbb{R}^\omega/\ell_\beta$.
 \end{proof}

Theorem \ref{embedEf} and Lemma \ref{embedlb} give the following   main theorem.
\begin{theorem}\label{embedlplq}
For $1 \le p <\infty$ and $U \in P(\omega)$, there is a Borel equivalence relation $E_U$ such that  for any $q>p$, $\mathbb{R}^\omega/\ell_p \le_B E_U \le_B \mathbb{R}^\omega/\ell_q$ and for $U, V \in P(\omega)$, we have
\[
U \subseteq^{*}V \Longleftrightarrow E_U \le_B E_V.
\]
\end{theorem}

 It is well known that every  Boolean algebra of size $\le \omega_1$ embeds into $P(\omega)/{\rm Fin}$. So under  CH (the continuum hypothesis), every partially ordered set of size at most continuum embeds into $P(\omega)/{\rm Fin}$, which implies  $P(\omega)/{\rm Fin}$ is the most complicated partially ordered set of  size at most continuum. On the other hand, we know that there is an antichain of size continuum in $P(\omega)/{\rm Fin}$ under ZFC, therefore there are continuum many incomparable Borel equivalence relations between $\mathbb{R}^\omega/\ell_p$ and $\mathbb{R}^\omega/\ell_q$ for $1 \le p <q <\infty$.
\begin{corollary}
For  $1 \le p <q<\infty$, there is a set of Borel equivalence relations
\[
\{\mathbf{E}_\xi:  \xi \in \{0,1\}^\omega\}
\]
such that $\mathbb{R}^\omega/\ell_p \le_B \mathbf{E}_{\xi} \le_B \mathbb{R}^\omega/\ell_q$, and for distinct $\xi,\zeta \in \{0,1\}^\omega$, we have $\mathbf{E}_{\xi}$ and $\mathbf{E}_{\zeta}$ are Borel incomparable.
\end{corollary}

A direct proof to show there are continuum many incomparable Borel equivalence relations between $\mathbb{R}^\omega/\ell_1$ and $\mathbb{R}^\omega/\ell_p$ for $p >1$ can be founded in \cite{DY}.

\section{Further remarks}

In this section we will compare the examples of equivalence relations in \cite{Ma} with what  we have constructed in the last section .

In \cite{Ma}, for $n<\omega$,  M\'atrai defined $t_n: (0,1] \to \mathbb{R}^+$ by:
\[
t_0(x)=1-\log_2(x), t_{n+1}(x)=1+\log_2(t_n(x)),
\]
for  $\eta \in [0,1)^{<\omega}$, denote $|\eta|$ the length of $\eta$,  let  $l_\eta: (0,1] \to \mathbb{R}^+$ defined by:
\[
l_\emptyset(x)=1, ~~l_\eta(x)=\prod_{i<|\eta|}t_i^{\eta_i}.
\]
For $\eta\neq \emptyset$, set $1/l_\eta(0)=0$, then $1/l_\eta$ is a continuous strictly increasing function on $[0,1]$ with $1/l_\eta(1)=1$.   M\'atrai showed  that  $(l_\eta(x^2))_{x \in (0,1]} \approx (l_\eta(x))_{x\in (0,1]}$ and for $\alpha \ge 1 $, $\mathbf{E}_{{\rm Id^\alpha}/l_\eta}$ is an equivalence relation.

Let $p_0=2$ and for $0<n<\omega$, $p_n=2^{p_{n-1}} $.
For $n<\omega$, let $s_n: (0,1/p_n] \to \mathbb{R}^+$ defined by:
\[
s_0(x)=-\log_2(x), s_{n+1}(x)=\log_2(s_n(x)),
\]
then extend $s_n$ to $(0, 1]$ by define $s_n(x)=s_n(1/p_n)$ for $x \ge 1/p_n$. We can show that  $s_n(x)\ge 1$ for $x \in (0,1]$ and $s_n(1/p_n)=1$ for $n<\omega$. For  $\eta \in [0,1)^{<\omega}$,  let  $l'_\eta: (0,1] \to \mathbb{R}^+$ defined by:
\[
l'_\emptyset(x)\equiv 1, ~~l'_\eta(x)=\prod_{i<|\eta|}s_i^{\eta_i}(x).
\]
Since $\lim_{x\to 0}l_\eta(x)/l'_\eta(x)=1$ and $l_\eta, l'_\eta$ are continuous on $[0,1]$ with $l_\eta(x), l'_\eta(x)>0$ for $x>0$, therefore $(l_\eta(x))_{x\in (0,1]} \approx (l'_\eta(x))_{x \in (0,1]}$. From $(l_\eta(x^2))_{x \in (0,1]} \approx (l_\eta(x))_{x\in (0,1]}$, we get $(l'_\eta(x^2))_{x \in (0,1]} \approx (l'_\eta(x))_{x\in (0,1]}$.
 For $\eta\neq \emptyset$,  set $1/l'_\eta(0)=0$, then  $1/l'_\eta$ is a continuous strictly  increasing function on $[0,1/2]$ with $1/l'_\eta(x)=1$ for $x \ge 1/2$.

Let $k_0(m)=2^{m+1}$ for each $m<\omega$ and for $i<\omega$, let
$ k_{i+1}(m)=2^{k_{i}(m)}$.

For $\emptyset \neq \eta \in [0,1)^{<\omega}$, let $j_0$ to be the minimum natural number with $\eta_{j_0} \neq 0$, set $k_m^\eta=k_{j_0}(m)$ for each $m<\omega$,  set $\delta_0^\eta=1/l'_\eta(1/2^{k_0^\eta})<1 $ and $\delta_m^\eta=l'_\eta(1/2^{k_{m-1}^\eta})/l'_\eta(1/2^{k_m^\eta})<1$ for $m>0$. Since
\[
\delta_m^\eta=\frac{(2^m)^{\eta_{j_0}}(m)^{\eta_{j_0+1}}\cdots (s_{|\eta|-1}(1/2^{k_{m-1}^\eta}))^{\eta_{|\eta|-1}}}{(2^{m+1})^{\eta_{j_0}}(m+1)^{\eta_{j_0+1}}\cdots (s_{|\eta|-1}(1/2^{k_{m}^\eta}))^{\eta_{|\eta|-1}}}
\]
for $m$ large enough, therefore $\lim_{m \to \infty}\delta_m^\eta=(1/2)^{\eta_{j_0}}$ and $0<\inf_m \delta_m^\eta \le \sup_m \delta_m^\eta <1$.
 Define $u_n^\eta$  by $u_0^\eta=1$ and for $n>0$,
\begin{equation*}
u_n^\eta=
\begin{cases}
\delta_m^\eta u_{n-1}^\eta, & n=k_m^\eta,\\
u_{n-1}^\eta, & {\rm otherwise}.
\end{cases}
\end{equation*}

Let $\varphi_\eta: [0,1] \to \mathbb{R}^+$ be continuous increasing with $\varphi_\eta(1/2^n)=u_n^\eta$. Let $x_0=1$ and for $n>0$, set $x_n=1/2^{k_{n-1}^\eta}$, then $\varphi_\eta(x_n)=1/l'_\eta(x_n)$ and $\varphi_\eta(x_{n+1}) \ge (\inf_m \delta_m^\eta) \varphi_\eta(x_{n})$ for each $n<\omega$, by Proposition \ref{funseq1}, we have $(\varphi_\eta(x))_{x\in(0,1]} \approx (1/l'_\eta(x))_{x\in(0,1]}$. Since $\varphi_\eta, 1/l'_\eta$ are continuous increasing on $[0,1]$ and they are equal at $x_n$ for each $n<\omega$, we have $\varphi_\eta(0)=1/l'_\eta(0)$, therefore $(\varphi_\eta(x))_{x\in[0,1]} \approx (1/l'_\eta(x))_{x\in[0,1]}\approx (1/l_\eta(x))_{x\in[0,1]}$, which implies $\mathbf{E}_{{\rm Id^\alpha}/l_\eta}$ belong to the equivalence relations we have constructed in the last section.

 Let $<_{lex}$ denote the lexicographical order, M\'atrai gave the following lemma.
\begin{lemma}[M\'{a}trai  \cite{Ma}, Corollary 30] \label{leta}
For every $1 \le \alpha <\beta<\infty$ and $\eta, \eta' \in [0,1)^{<\omega}$ with $\eta<_{lex} \eta'$,
\[
\mathbb{R}^\omega/\ell_\alpha <_B \mathbf{E}_{{\rm Id^\alpha}/l_\eta}<_B \mathbf{E}_{{\rm Id^\alpha}/l_{\eta'}}<_B \mathbb{R}^\omega/\ell_\beta.
\]
\end{lemma}
We would like to give a different proof here.

\begin{proof}
 From the analysis above and  Theorem \ref{embedEf}, we know that $\mathbb{R}^\omega/\ell_\alpha <_B \mathbf{E}_{{\rm Id^\alpha}/l_\eta}$. Since $\eta<_{lex} \eta'$, set $j_0$ to be the minimum $j$ with $\eta(j)<\eta'(j)$(if $|\eta|<|\eta'|$, we can extend $\eta$ to length $|\eta'|$ by setting $\eta(j)=0$ for $j \ge |\eta|$, we do the same if  $|\eta'|<|\eta|$). Let $f(x)=x^\alpha/l'_\eta(x)$ and $g(x)=x^\alpha/l'_\eta(x)\cdot l'_\eta(x)/l'_{\eta'(x)}$ for $x\in[0,1]$. For $i, m<\omega$, let $k_i(m)$ defined as above. Since
\[
l'_\eta(x)/l'_{\eta'(x)}=\frac{1}{(s_{j_0}(x))^{\eta'(j_0)-\eta(j_0)} \cdots }
\]
and $\lim_{x\to 0}l'_\eta(x)/l'_{\eta'(x)}=0$, there is $m_0<\omega$ such that for $x \le 1/2^{k_{j_0}(m_0)}$, $l'_\eta(x)/l'_{\eta'(x)}$ is strictly increasing with $l'_\eta(x)/l'_{\eta'(x)} \le 1/2$. For $x>0$, $l'_\eta(x)/l'_{\eta'(x)}>0$, therefore $l'_\eta(x)/l'_{\eta'(x)}$ is essentially increasing on $[0,1]$. Let $k_m= k_{j_0}(m_0+m)$ for each $m<\omega$.  Set $\delta_0=l'_\eta(1/2^{k_0})/l'_{\eta'}(1/2^{k_0})<1 $ and for $m>0$,
\[
\delta_m=l'_\eta(1/2^{k_m})l'_{\eta'}(1/2^{k_{m-1}})/l'_\eta(1/2^{k_{m-1}})l'_{\eta'}(1/2^{k_m})<1.
\]

It is easy to check that $\lim_{m \to \infty}\delta_m=(1/2)^{\eta'_{j_0}-\eta_{j_0}}$ and $0<\inf_m \delta_m \le \sup_m \delta_m <1$.
 Define $u_n$  by $u_0=1$ and for $n>0$,
\begin{equation*}
u_n=
\begin{cases}
\delta_m u_{n-1}, & n=k_m,\\
u_{n-1}, & {\rm otherwise}.
\end{cases}
\end{equation*}

Let $\varphi: [0,1] \to \mathbb{R}^+$ be continuous increasing with $\varphi(1/2^n)=u_n$. From Theorem \ref{embedEf}, we have $\mathbf{E}_{{\rm Id^\alpha}/l'_\eta}<_B\mathbf{E}_{{\rm Id^\alpha}\varphi/l'_\eta}$. Let $x_0=1$ and for $n>0$, set $x_n=1/2^{k_{n-1}}$, then $\varphi(x_n)=l'_\eta(x_n)/l'_{\eta'(x_n)}$ and   $\varphi(x_{n+1}) \ge (\inf_m \delta_m^\eta) \varphi(x_{n})$ for each $n<\omega$. By Proposition \ref{funseq1}, we have $(\varphi(x))_{x\in(0,1]} \approx (l'_\eta(x)/l'_{\eta'(x)})_{x\in(0,1]}$, since $\varphi, l'_\eta/l'_{\eta'}$ are continuous increasing on $[0, 1/2^{k_0}]$ and they are equal at $x_n$ for $0<n<\omega$, then $\varphi(0)=l'_\eta(0)/l'_{\eta'(0)}$ and $(\varphi(x))_{x\in[0,1]} \approx (l'_\eta(x)/l'_{\eta'(x)})_{x\in[0,1]}$, therefore $\mathbf{E}_{{\rm Id^\alpha}/l_\eta}=\mathbf{E}_{{\rm Id^\alpha}/l'_\eta}<_B\mathbf{E}_{{\rm Id^\alpha}/l'_{\eta'}}=\mathbf{E}_{{\rm Id^\alpha}/l_{\eta'}}$.

Since $\lim_{n\to \infty} \ell_\eta(1/2^{n})/\ell_\eta(1/2^{n+1})=1$,  using Lemma \ref{lb} we get $\mathbf{E}_{{\rm Id^\alpha}/l_\eta}\le_B \mathbb{R}^\omega/\ell_\beta$. As $\mathbf{E}_{{\rm Id^\alpha}/l_\eta}<_B \mathbf{E}_{{\rm Id^\alpha}/l_{\eta'}}$ for  $\eta<_{lex} \eta'$, it follows that the reduction is strictly.
\end{proof}

In the end, for $\eta, \eta' \in [0,1)^{<\omega}$ with $\eta<_{lex} \eta'$, we show that $P(\omega)/{\rm Fin}$ can be embedded into Borel equivalence relations between
$\mathbf{E}_{{\rm Id^\alpha}/l_\eta}$ and $\mathbf{E}_{{\rm Id^\alpha}/l_{\eta'}}$.

For $l<\omega$, let $a_l, I_l$ are defined as in Theorem \ref{embedEf}. For $i, m<\omega$, let $k_i(m)$ defined as above. Set $j_0$ to be the minimum $j$ with $\eta(j)<\eta'(j)$. Let $k_m=k_{j_0}(m)$ for each $m<\omega$ and $\delta<1$ satisfying $\delta>2^{\eta(j_0)-\eta'(j_0)}$, then $\log_2 1/\delta < \eta'(j_0)- \eta(j_0)$, for every $U \in P(\omega)$, set $u_U(0)=1$ and for each $n>0$, define $u_U(n)$ by:
\begin{equation*}
u_U(n)=
\begin{cases}
\delta u_U(n-1), & n=k_m, m \in I_l, l \in U,\\
u_U(n-1), & {\rm otherwise}.
\end{cases}
\end{equation*}
Then we have the following Theorem. This theorem together with Lemma \ref{leta} give a concrete description about Theorem \ref{embedlplq}.
\begin{theorem}
Let $\alpha \ge 1$ and $\eta, \eta' \in [0,1)^{<\omega}$ with $\eta<_{lex} \eta{'}$, for  each $U \in P(\omega)$ and $n< \omega$, let $u_U(n), I_n, a_n, k_n, \delta$ are defined as above. Let $\varphi_U: [0,1] \to \mathbb{R}^+$ be  continuous increasing with  $\varphi_U(1/2^n)=u_U(n)$ for each $n<\omega$. Set $f_U(x)=x^\alpha\varphi_U(x)/l_\eta(x)$ for $x\in [0,1]$. Then for each $U, V \in P(\omega)$,
\[
U\subseteq^{*}V\Longleftrightarrow \mathbf{E}_{f_U} \le_B \mathbf{E}_{f_V}
\]
and
\[
\mathbf{E}_{{\rm Id^\alpha}/l_\eta} \le_B \mathbf{E}_{f_U}<_B \mathbf{E}_{{\rm Id^\alpha}/l_{\eta'}}.
\]
\end{theorem}
\begin{proof}
From Theorem \ref{embedEf}, we have $\mathbf{E}_{f_U}$ is an equivalence relation and for each $U, V \in P(\omega)$,
\[
U\subseteq^{*}V\Longleftrightarrow \mathbf{E}_{f_U} \le_B \mathbf{E}_{f_V}
\]
and
\[
\mathbf{E}_{{\rm Id^\alpha}\varphi_\emptyset/l_\eta}  \le_B \mathbf{E}_{f_U}\le_B \mathbf{E}_{{\rm Id^\alpha}\varphi_{\omega}/l_\eta}.
\]

If $U=\emptyset$, then $\varphi_U(1/2^n)=1$ for $n<\omega$, since $\varphi_U$ is continuous increasing in $[0,1]$, it follows $\varphi_\emptyset \equiv 1$.

If $U=\omega$, by Lemma \ref{un} (i), we have  $u_\omega(k_n)=\delta^{n+1}$ for $n<\omega$, therefore  $\varphi_\omega(1/2^{k_n})=\delta^{n+1}$.

Let $x_n=1/2^{k_n}$, then $n= s_{j_0+1}(x_n)-1$, therefore
\begin{align*}
\varphi_\omega(x_n)& =\delta^{n+1}=\delta^{s_{j_0+1}(x_n)-1+1}=2^{  s_{j_0+1}(x_n)\log_2\delta}\\
&=(s_{j_0}(x_n))^{\log_2\delta}=1/s_{j_0}(x_n)^{\log_2 (1/\delta)}.
\end{align*}
Since $\varphi_\omega(x_{n+1})=\delta\varphi_\omega(x_{n})$ and $\varphi_\omega(x_{0})=\delta \varphi_\omega(1)$, by Proposition \ref{funseq1} we have
$(\varphi_\omega(x))_{x\in(0,1]} \approx (1/s_{j_0}(x)^{\log_2 (1/\delta)})_{x\in(0,1]}$. Note that $\varphi_\omega, 1/s_{j_0}^{\log_2 (1/\delta)}$ are
continuous increasing functions and they are equal at $x_n$ for each $n<\omega$, therefore $\varphi_\omega(0)=1/s_{j_0}(0)^{\log_2 (1/\delta)}=0$ and $(\varphi_\omega(x))_{x\in[0,1]} \approx (1/s_{j_0}(x)^{\log_2 (1/\delta)})_{x\in[0,1]} \approx (1/t_{j_0}(x)^{\log_2 (1/\delta)})_{x\in[0,1]}$.

Since ${\rm Id^\alpha}/l_{\eta}= {\rm Id^\alpha}\varphi_\emptyset/l_{\eta}$, we have $\mathbf{E}_{{\rm Id^\alpha}/l_\eta}\le_B \mathbf{E}_{f_U}$ for each $U \in P(\omega)$.

Let $\eta_0 \in [0,1)^{<\omega}$ be the same length with $\eta$, set $\eta_0$ by $\eta_0(j)=\eta(j)$ for $j\neq j_0$, and $\eta_0(j_0)=\eta(j_0)+\log_2(1/\delta)$, since $\log_2 1/\delta < \eta'(j_0)- \eta(j_0)$ by the choice of $\delta$, we know  $\eta_0 <_{lex} \eta'$. Since $(\varphi_\omega(x)/l_\eta(x))_{x\in[0,1]} \approx (1/l_\eta(x) \cdot 1/t_{j_0}(x)^{\log_2(1/\delta)})_{x\in[0,1]}=(1/l_{\eta_0}(x))_{x\in[0,1]}$ and $\mathbf{E}_{{\rm Id^\alpha}/l_{\eta_0}}<_B \mathbf{E}_{{\rm Id^\alpha}/l_{\eta'}}$ by Corollary \ref{leta}, therefore
$\mathbf{E}_{f_U}\le_B \mathbf{E}_{{\rm Id^\alpha}/l_{\eta_0}}<_B \mathbf{E}_{{\rm Id^\alpha}/l_{\eta'}}$.
\end{proof}

\section{Acknowledgements}

This paper is the main part of my PhD thesis, I would like to express my gratitude to my supervisor Longyun Ding for his consistent guidance and encouragement on my research work. I should also thank Xin Ma, Minggang Yu and Yukun Zhang for discussions in seminars. Special thanks to Jialiang He for providing me an example of an antichain of size continuum in $P(\omega)/{\rm Fin}$.

\bibliographystyle{amsplain}

\end{document}